\documentclass[10pt,twoside,a4paper,english,reqno]{amsart}
\usepackage{amscd}
\usepackage{caption}
\usepackage{a4wide}
\usepackage{amssymb,amsthm,amsmath,amsfonts,mathrsfs}
\usepackage{bm,latexsym,enumitem,dsfont,tabularx,graphicx,url}
\usepackage[utf8]{inputenc}
\usepackage{hyperref}
\usepackage{stmaryrd}
\usepackage{subcaption} 

\usepackage{algorithm}

\usepackage{algpseudocode}
\hypersetup{
    colorlinks=true,
    linkcolor=blue,
    filecolor=red,
    urlcolor=blue,
    citecolor=red,
}
\usepackage{nicefrac,microtype,upref,doi,footmisc}

\usepackage[svgnames]{xcolor}
\usepackage{tikz}
\usepackage{orcidlink}

\theoremstyle{plain}
\newtheorem{theorem}{Theorem}[section]

\newtheorem{lemma}[theorem]{Lemma}

\newtheorem{remark}[theorem]{Remark}
\newtheorem{assumption}[theorem]{Assumption}

\numberwithin{equation}{section}

\algrenewcommand\algorithmicrequire{\textbf{Require:}}
\algrenewcommand\algorithmicensure{\textbf{Ensure:}}
\algrenewcommand\algorithmicfor{\textbf{For}}
\algrenewcommand\algorithmicdo{\textbf{do}}
\algrenewcommand\algorithmicif{\textbf{If}}
\algrenewcommand\algorithmicthen{\textbf{then}}
\algrenewcommand\algorithmicend{\textbf{end}}

\title[A Galerkin FEM for the Fractional Calder\'on Problem]{A Galerkin Finite Element Method for the Fractional Calder\'on Problem}

\author[M. Dwivedi]{Mukul Dwivedi\textsuperscript{*}\,\orcidlink{0000-0002-5891-2788}}
\thanks{$^*$Corresponding author}
\address{Department of Mathematics, Saarland University, Saarbr\"ucken, Germany} 
\email{mukul.dwivedi@uni-saarland.de} 

\author[J. Railo]{Jesse Railo\,\orcidlink{0000-0001-9226-4190}}
\address{Computational Engineering, School of Engineering Sciences, Lappeenranta-Lahti University of Technology, Finland} 
\email{jesse.railo@lut.fi} 

\author[A. Rupp]{Andreas Rupp\,\orcidlink{0000-0001-5527-7187}}
\address{Department of Mathematics, Saarland University, Saarbr\"ucken, Germany} 
\email{andreas.rupp@uni-saarland.de} 


\newcommand{\R}{\mathbb{R}}

\newcommand{\dist}{\operatorname{dist}}

\begin{document}

\begin{abstract}
We propose a two-stage Galerkin finite element method for reconstructing the potential in the fractional Calder\'on problem from a single partial exterior measurement. The first stage recovers the forward solution by Tikhonov regularization. The second reconstructs the potential by stabilized least squares using a regularized approximation of the fractional Laplacian of the forward solution. We prove convergence of the reconstructed forward solution under explicit discretization and parameter conditions, and algebraic error estimates under an additional source condition. We also prove a logarithmic coefficient estimate on an interior subdomain without a source condition, for bounded nonnegative potentials with compact interior support and positive Sobolev regularity and for nonnegative nonzero exterior Dirichlet data. Numerical experiments in one and two dimensions illustrate reconstructions of continuous and discontinuous potentials from noisy data.
\end{abstract}

\maketitle

\section{Introduction}\label{sec:intro}
The classical Calder\'on problem asks whether one can determine the interior electrical conductivity of a medium from voltage and current measurements made only on the boundary.
In its conductivity form, one seeks to recover $\gamma$ in
\[
\nabla\!\cdot(\gamma\nabla u)=0, \qquad \text{in some domain }\Omega\subset \mathbb{R}^d,
\]
from the Dirichlet-to-Neumann (DN) map that sends boundary voltages to boundary currents \cite{calderon2006}.
After Calder\'on's seminal observation, global uniqueness was established in many settings, including the foundational result of Sylvester--Uhlmann in dimensions $d\ge 3$ \cite{sylvester1987}, the two-dimensional theory of Astala--P\"aiv\"arinta \cite{astala2006}, and partial data results such as \cite{kenig2007}. A broader account is given in \cite{uhlmann2014}. From the computational standpoint, Calder\'on-type inverse problems are a canonical example of instability. Even when uniqueness holds, the reconstruction map is highly ill-posed and regularization is essential \cite{engl1996regularization,kaipio2005statistical}.

In many applications, classical diffusion models are replaced by nonlocal operators that capture long-range interactions, anomalous transport, or jump processes.
A central example is the fractional Laplacian $(-\Delta)^s$ with $s\in(0,1)$, which admits several equivalent formulations (Fourier multipliers, singular integrals, extensions) \cite{kwasnicki2017equivalent} and is the negative generator of an isotropic $2s$-stable L\'evy process \cite{applebaum2009levy}.
Fractional and nonlocal PDEs are used in models of anomalous diffusion and transport \cite{metzler2000random}, in materials with nonlocal interactions, and in related areas; see, for instance, \cite{bucur2016nonlocal,lischke2020what,rosoton2016survey}.
In this setting, the natural boundary interaction is typically replaced by an exterior interaction, which leads to inverse problems driven by exterior data.

In this work, we consider the fractional Schr\"odinger equation
\begin{equation}\label{eq:fwd}
\begin{cases}
(-\Delta)^s u + qu = 0 &\text{in }\Omega,\\
u=f &\text{in }\Omega_e:=\mathbb{R}^d\setminus\overline{\Omega},
\end{cases}
\end{equation}
where $\Omega\subset\mathbb{R}^d$ is a bounded Lipschitz domain, $s\in(0,1)$ is fixed, and $q\in L^\infty(\Omega)$ is the unknown potential.
The solution $u\in H^s(\mathbb{R}^d)$ agrees with the exterior Dirichlet datum $f$ in $\Omega_e$. The corresponding DN map $\Lambda_q$ assigns to $f$ the exterior nonlocal flux of the solution and is defined by \eqref{eq:DN_pointwise_def}. Its weak formulation uses the bilinear form associated with $(-\Delta)^s$ \cite{ghosh2020calderon,salo2017fractional}.
We keep the analysis in general dimension $d\ge 1$, while the numerical experiments in this paper focus on $d\le 2$. 

A nonlocal analogue of the Calder\'on problem is to determine $q$ from (partial) knowledge of $\Lambda_q$. This was introduced for the fractional Schr\"odinger equation in \cite{ghosh2020calderon}, where global uniqueness was proved even when exterior data are supported in one nonempty open set and observations are made on another (possibly disjoint) open set. A key feature of the nonlocal setting is that $(-\Delta)^s$ couples $\Omega$ directly to $\Omega_e$, which leads to strong unique continuation and Runge approximation properties \cite{ghosh2020calderon,ruland2015unique}.
The single-measurement result in \cite{ghosh2020uniqueness} holds for continuous potentials when $0<s<1$ and for bounded potentials when $1/4\le s<1$, under the nonresonance assumption. The fractional Calder\'on theory has rapidly expanded to broader operators and geometries, including variable-coefficient nonlocal operators \cite{ghosh2017calderon}, lower-order nonlocal perturbations \cite{bhattacharyya2021inverse}, higher-order local perturbations \cite{covi2022higher}, time-dependent models \cite{kow2022wave}, unbounded domains via suitable Poincar\'e inequalities \cite{railo2023poincare}, and closed Riemannian manifolds \cite{FeizmohammadiKrupchykUhlmann2026}.

Despite its uniqueness, the problem remains severely ill-posed.
Logarithmic stability is essentially optimal, and exponential instability phenomena can occur \cite{ruland2018exponential,ruland2020low}, while refined single-measurement stability estimates are studied in \cite{Ruland2021stability}.
Consequently, robust reconstruction from noisy data requires regularization and careful discretization that respects the nonlocal forward operator.

On the numerical side, two main difficulties interact: the nonlocality of the fractional Laplacian typically produces a dense system matrix, meaning that most degrees of freedom interact with many others, so assembly and linear algebra are substantially more expensive than in local problems, while the inverse step amplifies noise and must be regularized. For the fractional forward model, many discretization strategies exist, including finite difference--quadrature schemes \cite{duo2018,huang2014numerical}, finite element methods for the integral fractional Laplacian \cite{acosta2017fractional,ainsworth2017adaptive,bonito2018numerical},
and approaches based on extensions or related operator-theoretic formulations \cite{caffarelli2007extension,hofreither2019unified,nochetto2015pde}. For inverse problems, one may use optimization methods, regularization methods, or monotonicity-based techniques for suitable classes of potentials \cite{harrach2019monotonicity,harrach2020monotonicity}.
Li \cite{li2025numerical} proposed a finite difference method with Tikhonov regularization for recovering a potential from one measurement. Li \cite{li2024simultaneous} also studied the simultaneous recovery of a potential and an internal source from two measurements. Rinaldo and Warma \cite{RinaldoWarma2026} developed two reconstruction methods for the fractional Laplacian and applied them to potential recovery with nonlocal Robin exterior conditions. A Bayesian formulation of the single-measurement fractional Calder\'on problem was studied in \cite{kow2025bayesian}.

The purpose of the present work is to provide a finite element realization of the unique continuation based single-measurement reconstruction principle of \cite{ghosh2020uniqueness}. Our method first recovers the forward solution \(u\) by Tikhonov regularization of an exterior restriction operator. For fixed meshes and bases, the discrete observation operator depends only on the geometry, the observation set, and the fractional order, so it can be assembled once and reused for different potentials and data. The function $u_{0,h}^\alpha\in V_{h,0}$ vanishes in $\Omega_e$, so its fractional Laplacian on $W$ is given by a nonsingular integral over $\Omega$. The second stage uses a regularized approximation of $(-\Delta)^s u$ in $\Omega$ to reconstruct $q$ by stabilized least squares, avoiding pointwise division by the reconstructed forward solution. We prove convergence of the reconstructed forward solution under explicit parameter conditions and derive error estimates for the forward solution, its fractional Laplacian in $\Omega$, and the potential under the stated assumptions. For nonnegative potentials and nonnegative nonzero exterior data, we prove a logarithmic coefficient error estimate on an interior subdomain under the stated regularity and support assumptions, without a spectral source condition. The numerical experiments in one and two space dimensions illustrate reconstructions of continuous and discontinuous potentials from noisy data.

The novelty of the present work lies not merely in the finite element realization of the continuous reconstruction principle of \cite{ghosh2020uniqueness}, but in the development and analysis of a regularized Galerkin reconstruction framework for the entire inverse problem. In contrast to the finite difference-based iterative coefficient reconstruction in \cite{li2025numerical}, our approach decouples the recovery of the forward solution from the identification of the potential, thereby avoiding repeated solutions of the parameter-dependent forward problem and the associated nonlinear optimization. The analysis explicitly quantifies the interplay between measurement noise, spatial discretization, domain truncation, and regularization. A further contribution is the regularized variational recovery of the fractional Laplacian combined with a stabilized coefficient reconstruction in a dedicated discontinuous finite element space, which avoids potentially unstable pointwise division. In particular, Theorem \ref{thm:stability2d} establishes a logarithmic error estimate for the discrete reconstructed potential under positivity, interior support, and Sobolev regularity assumptions, without requiring a spectral source condition. These results provide a quantitative link between continuous unique continuation theory and the convergence of a practically implementable finite element reconstruction method.

The paper is organized as follows. Section \ref{sec:cont_model} introduces the continuous forward problem, the DN map, and the reduction of the single-measurement inverse problem to an exterior unique continuation inversion. Section \ref{sec:recon} develops the discrete reconstruction method, and Section \ref{sec:numerics} presents numerical experiments. We conclude with a discussion in Section \ref{sec:conclusion}.

\section{Mathematical setting and single-measurement reconstruction}\label{sec:cont_model}

This section records the continuous model and the single-measurement reconstruction principle that motivates the discrete algorithm in Section~\ref{sec:recon}.
The weak formulation and the DN map below apply in every dimension $d\ge1$ \cite[Sec. 2]{ghosh2020uniqueness}. The single-measurement reconstruction in \cite[Thms. 1 and 2]{ghosh2020uniqueness} also applies in this range.

Let $s\in(0,1)$ and let $\Omega\subset\mathbb{R}^d$ be a bounded Lipschitz domain with exterior $\Omega_e$.
We write $H^s(\mathbb{R}^d)$ for the standard fractional Sobolev space.
For an open set $U\subset\mathbb{R}^d$, we use the restriction space equipped with the quotient norm \cite[Sec. 2]{ruland2020low}, i.e.,
\[
H^s(U):=\{u|_{U}:\ u\in H^s(\mathbb{R}^d)\}, \qquad \|u\|_{H^s(U)} := \inf \{ \|v\|_{H^s(\R^d)}\;:\; v|_{U} = u\}.
\]
We also use the interior space
\[
\widetilde H^s(\Omega):=\overline{C_c^\infty(\Omega)}^{\,H^s(\mathbb{R}^d)},
\]
whose elements can be identified with $H^s(\mathbb{R}^d)$-functions supported in $\overline{\Omega}$. The space $\widetilde H^s(\Omega)$ inherits the norm and inner product of $H^s(\mathbb R^d)$, which we denote by $\|\cdot\|_{H^s(\mathbb R^d)}$ and $(\cdot,\cdot)_{H^s(\mathbb R^d)}$, respectively (see \cite[Sec. 2.1]{ghosh2020uniqueness}). For an open set $W\subset\Omega_e$, we define $\widetilde H^s(W)$ analogously and denote its dual by $H^{-s}(W):=(\widetilde H^s(W))^\ast$. For sufficiently regular $u:\mathbb{R}^d\to\mathbb{R}$, the integral fractional Laplacian is given by
\begin{equation}\label{eq:fraclap_def}
(-\Delta)^s u(x)
:= c_{d,s}\,\mathrm{P.V.}\!\int_{\mathbb{R}^d}\frac{u(x)-u(y)}{|x-y|^{d+2s}}\,dy,
\end{equation}
with a normalization constant $c_{d,s}>0$ (see \cite{kwasnicki2017equivalent} and the references therein).  Associated with \eqref{eq:fraclap_def} is the bilinear energy form
\begin{equation}\label{eq:bilinear_a}
a(u,v) := \frac{c_{d,s}}{2}\int_{\mathbb{R}^d}\int_{\mathbb{R}^d}
\frac{\big(u(x)-u(y)\big)\big(v(x)-v(y)\big)}{|x-y|^{d+2s}}\,dx\,dy,
\end{equation}
which is continuous on $H^s(\mathbb{R}^d)\times H^s(\mathbb{R}^d)$ and satisfies
$\langle (-\Delta)^s u,v\rangle = a(u,v)$ in the distributional sense for suitable $u,v$; see \cite{dinezza2012hitchhiker,kwasnicki2017equivalent,rosoton2016survey}.
We equip $H^s(\mathbb R^d)$ with the equivalent inner product
\[
(v,z)_{H^s(\mathbb R^d)}:=(v,z)_{L^2(\mathbb R^d)}+a(v,z),\qquad v,z\in H^s(\mathbb R^d).
\]

Let $q\in L^\infty(\Omega)$ be real-valued. For an exterior datum $f\in H^s(\Omega_e)$, we also denote by $f$ a fixed extension in $H^s(\mathbb R^d)$. A weak solution of \eqref{eq:fwd} is a function $u\in H^s(\mathbb R^d)$ such that $u-f\in\widetilde H^s(\Omega)$ and
\begin{equation}\label{eq:weak_forward}
a(u,\phi) + \int_{\Omega} q\,u\,\phi\,dx = 0
\qquad \forall \phi\in \widetilde H^s(\Omega).
\end{equation}
As in \cite{ghosh2020uniqueness}, well-posedness is ensured by the following assumption.

\begin{assumption}[No zero Dirichlet eigenvalue]\label{ass:nonresonance}
The only $w\in \widetilde H^s(\Omega)$ satisfying
$a(w,\phi)+\int_{\Omega} q\,w\,\phi\,dx=0$ for all $\phi\in \widetilde H^s(\Omega)$ is $w\equiv 0$.
\end{assumption}
Under Assumption \ref{ass:nonresonance}, the problem \eqref{eq:weak_forward} has a unique solution $u\in H^s(\mathbb R^d)$ with $u-f\in\widetilde H^s(\Omega)$ \cite[Lem. 2.1]{ghosh2020uniqueness}.
The fractional DN map is the operator $\Lambda_q$ that assigns to $f$ the exterior nonlocal flux,
\begin{equation}\label{eq:DN_pointwise_def}
\Lambda_q f := \big((-\Delta)^s u\big)\big|_{\Omega_e}\in H^{-s}(\Omega_e),
\end{equation}
interpreted in the distributional sense. For $\phi\in C_c^\infty(\Omega_e)$, the weak DN pairing equals $a(u,\phi)=\langle(-\Delta)^s u,\phi\rangle$ because $\phi$ vanishes in $\Omega$. Thus \eqref{eq:DN_pointwise_def} gives the same measurements on exterior test functions \cite[Sec. 2.2]{ghosh2020uniqueness}.

Fix nonempty open sets $W_1,W_2\subset\Omega_e$.
In the partial data setting, one prescribes exterior inputs supported in $W_1$ and measures DN data on $W_2$.
In the single-measurement problem, we choose one exterior datum
\begin{equation}\label{eq:f-ass}
f\in C_c^\infty(W_1)\setminus\{0\},
\end{equation}
and the corresponding measurement is
\begin{equation}\label{eq:meas}
g := \Lambda_q f\big|_{W_2}.
\end{equation}

To connect directly with the discretization, we impose the geometric separation that makes the measurement pointwise.
Throughout the reconstruction procedure, we assume that
\begin{equation}\label{eq:W-sep}
W_2\subset\subset\Omega_e,
\qquad
\dist(W_2,\Omega)=:d_0>0.
\end{equation}

\begin{remark}
For simplicity of notation, we set $W_1=W_2=:W$. The arguments also apply when $W_1\ne W_2$.
\end{remark}

Since $f\in C_c^\infty(W)$, we may extend $f$ by $0$ to a globally smooth function on $\mathbb{R}^d$, which we denote by $\bar u_f$. We use the standard decomposition
\begin{equation}\label{eq:u-decomp}
u=\bar u_f+u_0,\qquad u_0:=u-\bar u_f\in\widetilde H^s(\Omega).
\end{equation}
Restricting \eqref{eq:DN_pointwise_def} to $W$ and using $u_0=0$ in $\Omega_e$, we obtain on $W$
\[
(-\Delta)^s u_0
= g - (-\Delta)^s \bar u_f,
\]
in the appropriate distributional sense.
This yields the preprocessed measurement
\begin{equation}\label{eq:mu-def}
\mu := g - \big((-\Delta)^s \bar u_f\big)\big|_{W}.
\end{equation}
In the noiseless continuous model, $\mu = ((-\Delta)^s u_0)|_{W}$.

The unique continuation step recovers $u_0\in \widetilde H^s(\Omega)$ from $\mu$ (see \cite[Theorem 2]{ghosh2020uniqueness}). This determines $u=\bar u_f+u_0$ and the product $qu=-(-\Delta)^s u$ in $\Omega$. The coefficient estimates in Section \ref{subsec:qconv_alpha_final} use the positive lower bound for $u$ on $\Omega'\subset \subset \Omega$ established in Lemma \ref{lem:positive-solution} under the stated positivity assumptions.

\begin{lemma}[Pointwise meaning of the preprocessed measurement]\label{lem:pointwise}
Suppose that \eqref{eq:W-sep} holds. If $v\in L^2(\Omega)$ is extended by $0$ outside $\Omega$, then for every $x\in W$ the quantity
\begin{equation}\label{eq:L-def}
(\mathcal Lv)(x):=(-\Delta)^s v(x) = -c_{d,s}\int_{\Omega}\frac{v(y)}{|x-y|^{d+2s}}\,dy
\end{equation}
is well-defined and $\mathcal Lv$ is smooth on $W$.
Moreover, one has the bound
\begin{equation}\label{eq:pointwise-bound}
\|\mathcal Lv\|_{L^\infty(W)}\le C(d,s,d_0,\Omega)\,\|v\|_{L^2(\Omega)},
\end{equation}
and for every integer $m\ge 0$ there exists $C_m>0$ such that
\begin{equation}\label{eq:L_Hm_bound}
\|\mathcal L v\|_{H^{m}(W)} \le C_m\,\|v\|_{L^2(\Omega)}\qquad \forall v\in L^2(\Omega).
\end{equation}
The same integral defines $\mathcal Lv$ smoothly in a neighbourhood of $\overline W$. In particular, $\mu=\mathcal Lu_0$ and $g=\mu+(-\Delta)^s\bar u_f$ have smooth representatives there.
\end{lemma}

\begin{proof}
Fix $x\in W$.
Since $\dist(x,\Omega)\ge d_0$, we have $|x-y|\ge d_0$ for all $y\in\Omega$. Hence, $|x-y|^{-d-2s}\le d_0^{-d-2s}$ and
\[
\int_{\Omega} \frac{|v(y)|}{|x-y|^{d+2s}}\,dy
\le d_0^{-d-2s}\int_{\Omega} |v(y)|\,dy
\le d_0^{-d-2s}|\Omega|^{1/2}\,\|v\|_{L^2(\Omega)}.
\]
This proves absolute convergence and \eqref{eq:pointwise-bound}.
Continuity of $\mathcal Lv$ on $W$ follows from dominated convergence since $x\mapsto |x-y|^{-d-2s}$ is smooth on $W$ for each fixed $y\in\Omega$ and is uniformly dominated by $d_0^{-d-2s}$. For each multi-index $\beta\in \mathbb N_0^d$ with $|\beta|\le m$, there exists a constant
$C_\beta>0$ such that
\(
|\partial_x^\beta k(x,y)| \le C_\beta\, d_0^{-d-2s-|\beta|}
\text{ for all } x\in W,\ y\in \Omega,
\)
where $k(x,y):=|x-y|^{-d-2s}$. Differentiating under the integral sign gives
\[
\partial_x^\beta (\mathcal Lv)(x)
=
-c_{d,s}\int_{\Omega} v(y)\,\partial_x^\beta k(x,y)\,dy.
\]
Using Cauchy--Schwarz in $y$, we obtain
\[
|\partial_x^\beta (\mathcal Lv)(x)|
\le
c_{d,s}
\Big(\int_{\Omega} |v(y)|^2\,dy\Big)^{1/2}
\Big(\int_{\Omega} |\partial_x^\beta k(x,y)|^2\,dy\Big)^{1/2}
\le C \|v\|_{L^2(\Omega)}.
\]
Choose a bounded open set $U$ with $\overline W\subset U\subset\subset\Omega_e$ and $\chi\in C_c^\infty(U)$ such that $\chi=1$ near $\overline W$. The same derivative estimates on $U$ give
\[
\|\mathcal Lv\|_{H^m(W)}\le\|\chi\mathcal Lv\|_{H^m(\mathbb R^d)}\le C_m\|v\|_{L^2(\Omega)}.
\]
Here $\chi\mathcal Lv$ is extended by zero outside $U$. This proves \eqref{eq:L_Hm_bound}.
\end{proof}

The operator $\mathcal L$ in \eqref{eq:L-def} is injective and has a dense range in $H^{-s}(W)$ by the exterior unique continuation property for the fractional Laplacian for any open $W \subset \Omega_e$, see \cite{ghosh2020calderon,ruland2015unique}.
In our setting, we additionally assumed that \(W\) is bounded.
This boundedness is used in Lemma \ref{lem:pointwise} to obtain the \(L^2(W)\)-based smoothing estimates and, following the cutoff argument in \cite[Lemma 2.2]{ghosh2020uniqueness}, it is also the natural hypothesis for compactness of \(\mathcal L\).
Under the separation assumption \eqref{eq:W-sep}, $\mathcal L:\widetilde H^s(\Omega)\to H^{-s}(W)$ is also compact (indeed, smooth), which is the fundamental mechanism behind the ill-posedness of $\mathcal L u_0|_{W} = \mu$ and the need for regularization \cite{ghosh2020uniqueness}. 
In Section~\ref{sec:recon}, we develop a fully discrete regularized inversion of $\mathcal L u_0|_{W} = \mu$ and a stable coefficient recovery strategy.

\section{Discrete unique continuation and reconstruction of the potential}\label{sec:recon}
\subsection{Discrete setting: truncation domain, finite element spaces, and discrete norms}\label{subsec:trunc-fe}

To perform computations, we truncate $\mathbb R^d$ to a bounded domain
\[
\Omega_R:=(-R,R)^d,\qquad \overline{\Omega}\cup W\subset\Omega_R,
\]
where $R>0$ is chosen to be sufficiently large.

The mesh $\mathcal T_h$ is a conforming, shape-regular, quasi-uniform simplicial mesh of $\Omega_R$ with mesh regularity constants independent of $h$ and $R$. Throughout the analysis, $h$ is a common upper bound for the element diameters of the meshes used below. The space $V_h\subset H^1(\Omega_R)$ consists of continuous piecewise polynomials of fixed degree $p\ge1$ on $\mathcal T_h$. We define the interior subspace:
\begin{equation}\label{eq:Vh0}
V_{h,0}:=\{v_h\in V_h:\ v_h=0\ \text{on}\ \Omega_R\setminus\Omega\}.
\end{equation}
\begin{assumption}[Mesh fitted to $\Omega$]\label{ass:fitted_mesh}
The domain $\Omega$ is polygonal or polyhedral for $d\ge2$, and for all sufficiently small $h$, $\partial\Omega$ is a union of faces of $\mathcal T_h$. Thus each closed element lies in $\overline\Omega$ or in $\overline{\Omega_R}\setminus\Omega$.
\end{assumption}
It follows from \eqref{eq:Vh0} that the restrictions of functions in $V_{h,0}$ belong to $H_0^1(\Omega)$. Their zero extensions belong to $\widetilde H^s(\Omega)$ for every $0<s<1$. We denote the nodal basis of $V_{h,0}$ by $\{\varphi_i\}_{i=1}^{N_0}$, where $N_0=\dim V_{h,0}$. For the extension $\bar u_f$ in \eqref{eq:u-decomp}, we define
\begin{equation}\label{eq:uhf-def}
u_{h,f}:=I_h \bar u_f\in V_h,
\end{equation}
where $I_h$ is the nodal interpolant on $\Omega_R$. Since $\bar u_f=0$ on $\partial\Omega_R$, the interpolant has zero trace there and we identify $u_{h,f}$ with its zero extension in $H^1(\mathbb R^d)$. For sufficiently small $h$, the standard nodal $L^2$ and $H^1$ estimates and interpolation give
\begin{equation}\label{eq:uhf-err}
\|\bar u_f-u_{h,f}\|_{H^s(\mathbb R^d)}\le Ch^{p+1-s}\|\bar u_f\|_{W^{p+1,\infty}(\mathbb R^d)}.
\end{equation}
The constant \(C\) is independent of $h$ and $R$, since the interpolation error is supported in a fixed neighbourhood of $\operatorname{supp}\bar u_f$.

We assume that $W$ admits a conforming, shape-regular, quasi-uniform simplicial mesh $\mathcal T_h^W$ of mesh size at most $h$. The space $S_h^W\subset C(\overline W)$ consists of continuous piecewise polynomials of fixed degree $p_W\ge1$, with $p_W+1>d/2$. Its nodal basis is $\{\vartheta_k\}_{k=1}^{N_W}$, with nodes $\{x_k\}_{k=1}^{N_W}\subset\overline W$ and $N_W=\dim S_h^W$. We define the data space
\[
Y:=\mathbb R^{N_W},\qquad\quad
\langle y,z\rangle_Y := y^\top \mathbf W z,\qquad\quad
\|y\|_Y^2 := y^\top \mathbf W y,
\]
where \(\mathbf W\in\mathbb R^{N_W\times N_W}\) with
\(
\mathbf W_{ij}:=\int_W \vartheta_i(x)\vartheta_j(x)\,dx
\), and \(\mathbf W\) is symmetric positive definite (SPD).
For every $\psi\in C(\overline W)$, we define its sampled vector by
\begin{equation}\label{eq:psi_W_def}
\psi_W := (\psi(x_1),\dots,\psi(x_{N_W}))^\top\in Y.
\end{equation}
We also define the nodal interpolant
\[
I_h^W\psi := \sum_{k=1}^{N_W}\psi(x_k)\vartheta_k \in S_h^W,
\]
and the discrete observation bilinear form by
\begin{equation}\label{eq:QhW_def}
Q_h^W(\phi,\psi):=\phi_W^\top \mathbf W \psi_W
=\int_W I_h^W\phi(x)\,I_h^W\psi(x)\,dx,
\qquad \phi,\psi\in C(\overline W).
\end{equation}
Thus, for every $\psi\in C(\overline W)$, we have
\begin{equation}\label{eq:Yh_norm_as_quad}
\|\psi_W\|_Y^2 = Q_h^W(\psi,\psi)=\|I_h^W\psi\|_{L^2(W)}^2.
\end{equation}

\begin{assumption}[Order-$m$ consistency of the discrete observation inner product on \(W\)]\label{ass:quad_order_m}
Let \(\mathcal T_h^W\) be a shape-regular, quasi-uniform partition of \(W\), and let
\(S_h^W\), \(I_h^W\), and \(\mathbf W\) be defined as above.
We assume that there exist an integer $m$ with $d/2<m\le p_W+1$ and a constant $C>0$, independent of $h$, such that
\begin{equation}\label{eq:obs_interp_error}
\|I_h^W\phi-\phi\|_{L^2(W)}\le Ch^m\|\phi\|_{H^m(W)}\qquad\forall\phi\in H^m(W).
\end{equation}
\end{assumption}
Since \(m>d/2\), the nodal values \(\phi(x_k)\) are well defined for \(\phi\in H^m(W)\) by
Sobolev embedding. Moreover, Assumption~\ref{ass:quad_order_m} implies
\begin{equation}\label{eq:quad_error}
\left|
\int_W \phi(x)\psi(x)\,dx - Q_h^W(\phi,\psi)
\right|
\le C h^{m}\|\phi\|_{H^{m}(W)}\|\psi\|_{H^{m}(W)}
\qquad \forall\,\phi,\psi\in H^{m}(W).
\end{equation}
Indeed,
\[
\begin{aligned}
\left|
\int_W \phi\psi\,dx - Q_h^W(\phi,\psi)
\right|
&=
\left|
\int_W \phi\psi\,dx - \int_W I_h^W\phi\, I_h^W\psi\,dx
\right| \\
&\le
\left|
\int_W (\phi-I_h^W\phi)\psi\,dx
\right|
+
\left|
\int_W I_h^W\phi\,(\psi-I_h^W\psi)\,dx
\right| \\
&\le
\|\phi-I_h^W\phi\|_{L^2(W)}\|\psi\|_{L^2(W)}
+
\|I_h^W\phi\|_{L^2(W)}\|\psi-I_h^W\psi\|_{L^2(W)}.
\end{aligned}
\]
Using Assumption~\ref{ass:quad_order_m} and
\[
\|I_h^W\phi\|_{L^2(W)}
\le
\|\phi\|_{L^2(W)}+\|I_h^W\phi-\phi\|_{L^2(W)}
\le C\|\phi\|_{H^m(W)},
\]
we obtain \eqref{eq:quad_error}.

\begin{lemma}[Convergence of the discrete data norm]\label{lem:Y_to_L2_Hm}
Assume $\mathrm{dist}(W,\Omega)=d_0>0$ and Assumption~\ref{ass:quad_order_m} for some integer $m> d/2$.
Then for every $v\in \widetilde H^s(\Omega)$,
\begin{equation}\label{eq:Y_to_L2_conv_Hm}
\|(\mathcal L v)_W\|_{Y}^2 \longrightarrow \|\mathcal L v\|_{L^2(W)}^2
\qquad \text{as } h\to 0,
\end{equation}
and there exists $C>0$ independent of $h$ such that
\begin{equation}\label{eq:Y_to_L2_rate_Hm}
\big|\,\|(\mathcal L v)_W\|_{Y}^2-\|\mathcal L v\|_{L^2(W)}^2\,\big|
\le C\,h^{m}\,\|v\|_{L^2(\Omega)}^2
\le C\,h^{m}\,\|v\|_{H^s(\mathbb R^d)}^2.
\end{equation}
\end{lemma}
\begin{proof}
Equations \eqref{eq:Yh_norm_as_quad} and \eqref{eq:quad_error}, applied with $\phi=\psi=\mathcal Lv$, bound the left-hand side of \eqref{eq:Y_to_L2_rate_Hm} by $Ch^m\|\mathcal Lv\|_{H^m(W)}^2$. Estimate \eqref{eq:L_Hm_bound} and the inequality $\|v\|_{L^2(\Omega)}\le\|v\|_{H^s(\mathbb R^d)}$ give \eqref{eq:Y_to_L2_rate_Hm}. Equation \eqref{eq:Y_to_L2_conv_Hm} follows as $h\to0$.
\end{proof}


Define $L_h:V_{h,0}\to Y$ by sampling $\mathcal Lv_h$ at $\{x_k\}_{k=1}^{N_W}\subset\overline W$:
\[
(L_h v_h)_k := (\mathcal{L}v_h)(x_k)= -c_{d,s}\int_{\Omega} \frac{v_h(y)}{|x_k-y|^{d+2s}}\,dy.
\]
Writing $v_h=\sum_{i=1}^{N_0} v_i\varphi_i$, where $\{\varphi_i\}_{i=1}^{N_0}$ is the nodal basis of $V_{h,0}$, we obtain the dense matrix representation
\begin{equation}\label{eq:B-matrix}
L_h v_h = B\mathbf v,\qquad
B\in\mathbb{R}^{N_W\times N_0},\qquad
B_{k i} := -c_{d,s}\int_{\Omega} \frac{\varphi_i(y)}{|x_k-y|^{d+2s}}\,dy, \quad \mathbf v = (v_1,\dots,v_{N_0})^\top.
\end{equation}
Since $|x_k-y|\ge d_0$ and $\varphi_i$ is polynomial on each element, the integrand in \eqref{eq:B-matrix} is smooth on each element. The entries of $B$ can therefore be approximated by elementwise quadrature. Let $g_W^\delta\in Y$ be the vector of noisy measurements satisfying
\begin{equation}\label{eq:noise}
\|g^\delta_W-g_W\|_Y\le\delta,\qquad\qquad g_W:=(g(x_1),\dots,g(x_{N_W}))^\top\in Y.
\end{equation}
We preprocess the measurements using the known smooth extension $\bar u_f$ as follows
\begin{equation}\label{eq:mu-vector}
\mu_W^\delta:=g_W^\delta-\mathcal L_{h,R}\bar u_f\in Y,
\end{equation}
where
\[
(\mathcal L_{h,R}\bar u_f)_k:=c_{d,s}\operatorname{P.V.}\!\int_{\Omega_R}\frac{\bar u_f(x_k)-\bar u_f(y)}{|x_k-y|^{d+2s}}\,dy,\qquad k=1,\dots,N_W.
\]
For a node $x_k\in\overline W\cap\Omega_R$, the ball $B_{r_k}(x_k)$ with $r_k=\tfrac12\operatorname{dist}(x_k,\partial\Omega_R)$ lies in $\Omega_R$. Taylor's formula gives
\[
\bar u_f(x_k)-\bar u_f(y)=-\nabla\bar u_f(x_k)\cdot(y-x_k)+\mathcal O(|y-x_k|^2),\qquad y\in B_{r_k}(x_k).
\]
The linear term divided by $|x_k-y|^{d+2s}$ is odd about $x_k$ and integrates to zero on every concentric annulus. The remaining singularity is bounded by $C|x_k-y|^{2-d-2s}$ and is integrable because $\int_0^{r_k}t^{1-2s}\,dt<\infty$ for $s<1$. For a node $x_k\in\overline W\cap\partial\Omega_R$, the integrand vanishes near $x_k$ because $\bar u_f\in C_c^\infty(\Omega_R)$.

\subsection{Truncated fractional energy and truncation error}\label{subsec:trunc-energy}
We define the truncated bilinear form for \(u,v\in H^s(\mathbb R^d)\) whose support is contained in \(\overline{\Omega_R}\) by
\[
a_R(u,v):=\frac{c_{d,s}}{2}\iint\limits_{\Omega_R\times\Omega_R}\frac{(u(x)-u(y))(v(x)-v(y))}{|x-y|^{d+2s}}\,dx\,dy
+c_{d,s}\int\limits_{\Omega_R}u(x)v(x)\Big(\int\limits_{\mathbb{R}^d\setminus\Omega_R}\frac{dy}{|x-y|^{d+2s}}\Big)\,dx.
\]
For $u,v\in H^s(\mathbb R^d)$ supported in $\overline{\Omega_R}$, splitting the full-space integral gives $a_R(u,v)=a(u,v)$. In particular, the norm in \eqref{eq:norm-sh} satisfies $\|v_h\|_{s,R,h}=\|v_h\|_{H^s(\mathbb R^d)}$ for every $v_h\in V_{h,0}$.

\begin{lemma}[Tail integral bound]\label{lem:tail-1d}
For $0<\rho<R$ and $x\in(-\rho,\rho)^d$, there holds
\begin{equation}\label{eqn:est_tail_term}
\int_{\mathbb R^d\setminus\Omega_R}|x-y|^{-d-2s}\,dy\le\frac{|\mathbb S^{d-1}|}{2s}(R-\rho)^{-2s}.
\end{equation}
Here $|\mathbb S^{d-1}|$ denotes the surface measure of the unit sphere in $\mathbb R^d$. Consequently, if $\Omega\subset(-\rho,\rho)^d$, then for every $v\in L^2(\Omega)$ extended by zero,
\begin{equation}\label{eq:trunc-err-1d}
\int_{\Omega_R}
v(x)^2
\left(
\int_{\mathbb R^d\setminus \Omega_R}\frac{dy}{|x-y|^{d+2s}}
\right)\,dx
\le
C_{d,s}(R-\rho)^{-2s}\|v\|_{L^2(\Omega)}^2.
\end{equation}
\end{lemma}

\begin{proof}
For $x\in(-\rho,\rho)^d$, the ball $B(x,R-\rho)$ lies in $\Omega_R$. Integration in polar coordinates over $\mathbb R^d\setminus B(x,R-\rho)$ gives \eqref{eqn:est_tail_term}. Multiplication by $v(x)^2$ and integration over $\Omega$ give \eqref{eq:trunc-err-1d}.
\end{proof}

\subsubsection{Deterministic consistency of the preprocessed data and total consistency error}
The sampled vector of $\mu$ in \eqref{eq:mu-def} is $\mu_W:=(\mu(x_1),\dots,\mu(x_{N_W}))^\top\in Y$. The deterministic consistency error is defined by
\[
\eta(h,R):=\bigl\|\bigl((-\Delta)^s\bar u_f\bigr)_W-\mathcal L_{h,R}\bar u_f\bigr\|_Y.
\]
By \eqref{eq:noise} and \eqref{eq:mu-vector},
\begin{equation}\label{eq:delta_etabound}
\|\mu_W^\delta-\mu_W\|_Y\le\delta+\eta(h,R).
\end{equation}
The zero extension satisfies $\operatorname{supp}\bar u_f=\operatorname{supp}f$. We choose $\rho_f>0$ such that $\operatorname{supp}f\subset(-\rho_f,\rho_f)^d$ and take $R>\rho_f$. At every node with $\bar u_f(x_k)\ne0$,
\[
(-\Delta)^s\bar u_f(x_k)-(\mathcal L_{h,R}\bar u_f)_k=c_{d,s}\bar u_f(x_k)\int_{\mathbb R^d\setminus\Omega_R}|x_k-y|^{-d-2s}\,dy.
\]
At all other nodes this difference is zero. For $z=(z_k)_{k=1}^{N_W}\in Y$, the nodal basis $\{\vartheta_k\}_{k=1}^{N_W}$ of the fixed-degree space $S_h^W$ satisfies
\[
\|z\|_Y=\Big\|\sum_{k=1}^{N_W}z_k\vartheta_k\Big\|_{L^2(W)}\le C|W|^{1/2}\max_k|z_k|.
\]
The bound \eqref{eqn:est_tail_term} with $\rho=\rho_f$ therefore gives
\[
\eta(h,R)\le C|W|^{1/2}\|\bar u_f\|_{L^\infty(\mathbb R^d)}(R-\rho_f)^{-2s}.
\]
The constant is independent of $h$ and $R$, so $\eta(h,R)\to0$ as $R\to\infty$ for fixed $W$ and $f$. If the known tail in the preceding identity is included at each node with $\bar u_f(x_k)\ne0$, then $\eta(h,R)=0$ for fixed $R$.

\subsection{Recovery of \texorpdfstring{\(u\)}{u} from the discrete Tikhonov functional}\label{subsec:tikhonov}
On $V_{h,0}$, we use the consistent discrete norm
\begin{equation}\label{eq:norm-sh}
\|v_h\|_{s,R,h}^2 := a_R(v_h,v_h) + \|v_h\|_{L^2(\Omega_R)}^2.
\end{equation}
Let $A_0,M_0\in\mathbb{R}^{N_0\times N_0}$ be the interior stiffness and mass matrices defined by
\begin{equation}\label{eq:A0M0}
(A_0)_{ij}:=a_R(\varphi_i,\varphi_j),\qquad (M_0)_{ij}:=(\varphi_i,\varphi_j)_{L^2(\Omega_R)}.
\end{equation}
Then, for $v_h=\sum\limits_{i=1}^{N_0} v_i\varphi_i$,
\[
\|v_h\|_{s,R,h}^2 = \mathbf v^\top S\mathbf v,
\]
where $S:=A_0+M_0.$

For $\alpha>0$, define the discrete Tikhonov functional $J_{h,\alpha}:\mathbb{R}^{N_0}\to\mathbb{R}$ by
\begin{equation}\label{eq:J}
J_{h,\alpha}(\mathbf v) := \|B\mathbf v-\mu^\delta_W\|_Y^2 + \alpha\, \mathbf v^\top S \mathbf v
= (B\mathbf v-\mu^\delta_W)^\top \mathbf W(B\mathbf v-\mu^\delta_W) + \alpha \mathbf v^\top S \mathbf v.
\end{equation}
Here $B$ is the observation matrix defined in \eqref{eq:B-matrix}. Let $\{\varphi_i\}_{i=1}^{N_0}$ be a nodal basis of $V_{h,0}$. For any coefficient vector $\mathbf v\in\mathbb{R}^{N_0}$, we set
\[
\mathcal R_h\mathbf v := v_h := \sum_{i=1}^{N_0} v_i\,\varphi_i\in V_{h,0}.
\]
Let $\mathbf v_h^\alpha$ denote the unique minimizer of $J_{h,\alpha}$ and set
\[
u_{0,h}^\alpha:=\mathcal R_h\mathbf v_h^\alpha\in V_{h,0},
\qquad
u_h^\alpha:=u_{h,f}+u_{0,h}^\alpha.
\]

\begin{theorem}[Existence, uniqueness, and normal equations]\label{thm:min}
For every $\alpha>0$, the discrete Tikhonov functional \eqref{eq:J} admits a unique minimizer $\mathbf v_h^\alpha\in\mathbb{R}^{N_0}$. Moreover, $\mathbf v_h^\alpha$ is the unique solution of the SPD linear system
\begin{equation}\label{eq:normal}
(B^\top \mathbf W B + \alpha S)\, \mathbf v_h^\alpha = B^\top \mathbf W\,\mu^\delta_W .
\end{equation}
\end{theorem}

\begin{proof}
For $v_h=\mathcal R_h\mathbf v$, equations \eqref{eq:norm-sh} and \eqref{eq:A0M0} give $\mathbf v^\top S\mathbf v=\|v_h\|_{H^s(\mathbb R^d)}^2$. Thus $S$ is symmetric positive definite, and $\mathbf v^\top B^\top\mathbf W B\mathbf v=\|B\mathbf v\|_Y^2\ge0$. For $\alpha>0$, the matrix $B^\top\mathbf W B+\alpha S$ is therefore symmetric positive definite. The derivative of \eqref{eq:J} vanishes exactly when \eqref{eq:normal} holds, and strict convexity gives the unique minimizer.
\end{proof}

\subsection{Error analysis of the approximation of \texorpdfstring{\(u\)}{u}}\label{subsec:state-conv}
The continuous Tikhonov functional for the exact data $\mu=\mathcal Lu_0$ is defined by
\begin{equation}\label{eq:Jcont}
J_\alpha(v):=\|\mathcal{L}v-\mu\|_{L^2(W)}^2 + \alpha \|v\|_{ H^s(\R^d)}^2,
\qquad v\in \widetilde H^s(\Omega).
\end{equation}
For $\alpha>0$, the functional \eqref{eq:Jcont} is weakly lower semicontinuous, coercive, and strictly convex. Its unique minimizer is
\begin{equation}\label{eq:u0alpha}
u_0^\alpha := \operatorname*{argmin}_{v\in \widetilde H^s(\Omega)} J_\alpha(v).
\end{equation}
Set $u^\alpha:=\bar u_f+u_0^\alpha$.

We next establish an error estimate for the reconstructed forward solution $u_h^\alpha$ in $H^s(\mathbb R^d)$. The first step is to compare the discrete reconstruction for exact sampled data $\mu_W$ with the continuous minimizer $u_0^\alpha$ given by \eqref{eq:u0alpha}. A direct projection estimate requires a higher Sobolev bound for $u_0^\alpha$ with explicit dependence on $\alpha$, which we obtain through the adjoint equation. The following estimates hold under \eqref{eq:W-sep} and Assumption \ref{ass:fitted_mesh} for every fixed $\beta\in(0,\min\{s,1/2\})$. 

The standard compact embedding $\widetilde H^s(\Omega)\hookrightarrow L^2(\Omega)$, together with the boundedness established in Lemma \ref{lem:pointwise}, shows that $\mathcal L:\widetilde H^s(\Omega)\to L^2(W)$ is bounded and compact, and its injectivity follows from \cite[Lem. 2.2]{ghosh2020uniqueness}. The spaces $\widetilde H^s(\Omega)$ and $L^2(W)$ are Hilbert spaces, so $\mathcal L$ admits a unique bounded Hilbert adjoint $\mathcal L^*:L^2(W)\to\widetilde H^s(\Omega)$ satisfying
\[
(\mathcal Lv,y)_{L^2(W)}=(v,\mathcal L^*y)_{H^s(\mathbb R^d)},\qquad v\in\widetilde H^s(\Omega),\quad y\in L^2(W).
\]
For $y\in L^2(W)$, we set $z:=\mathcal L^*y\in\widetilde H^s(\Omega)$. The definition of the Hilbert adjoint and Fubini's theorem applied to \eqref{eq:L-def} give
\begin{equation}\label{eq:disc-adjoint-weak}
\begin{aligned}
a(z,\phi)+(z,\phi)_{L^2(\Omega)} =(z,\phi)_{H^s(\mathbb R^d)}=(y,\mathcal L\phi)_{L^2(W)}=(G_y,\phi)_{L^2(\Omega)},\qquad \phi\in\widetilde H^s(\Omega),
\end{aligned}
\end{equation}
where \(G_y(x):=-c_{d,s}\int_W\frac{y(\xi)}{|x-\xi|^{d+2s}}\,d\xi\) for \(x\in\Omega\). The separation condition \eqref{eq:W-sep} gives $\|G_y\|_{L^2(\Omega)}\le C\|y\|_{L^2(W)}$. The choice $\phi=z$ in \eqref{eq:disc-adjoint-weak} yields $\|z\|_{H^s(\mathbb R^d)}\le\|G_y\|_{L^2(\Omega)}$, and hence $\|G_y-z\|_{L^2(\Omega)}\le C\|y\|_{L^2(W)}$. The function $z$ therefore solves $(-\Delta)^s z=G_y-z$ weakly in $\Omega$ with zero exterior data. The regularity theorem and the Besov-to-Sobolev embedding in \cite[Thm. 1.1 and Lem. 2.2]{BorthagarayNochetto2023} imply
\begin{equation}\label{eq:disc-adjoint-regularity}
\|z\|_{H^{s+\beta}(\mathbb R^d)}\le C\|G_y-z\|_{L^2(\Omega)}\le C\|y\|_{L^2(W)}.
\end{equation}
We use this additional regularity to estimate the projection error in $H^s(\mathbb R^d)$. The space $V_{h,0}$ contains the continuous piecewise affine space with zero boundary values because $p\ge1$. Let $\Pi_h:\widetilde H^s(\Omega)\to V_{h,0}$ be the orthogonal projection with respect to $(\cdot,\cdot)_{H^s(\mathbb R^d)}$.  The interpolation estimate in \cite[Sec. 4.1]{BorthagarayLiNochetto2024} and the best approximation property of $\Pi_h$, together with \eqref{eq:disc-adjoint-regularity}, give
\[
\|(I-\Pi_h)z\|_{H^s(\mathbb R^d)}=\inf_{\chi_h\in V_{h,0}}\|z-\chi_h\|_{H^s(\mathbb R^d)}\le Ch^\beta\|z\|_{H^{s+\beta}(\mathbb R^d)}\le Ch^\beta\|y\|_{L^2(W)}.
\]
The last estimate holds for $z=\mathcal L^*y$ with a constant independent of $y$, so its supremum over $\|y\|_{L^2(W)}=1$ bounds the operator norm of $(I-\Pi_h)\mathcal L^*$. The orthogonality of $\Pi_h$ gives $((I-\Pi_h)\mathcal L^*)^*=\mathcal L(I-\Pi_h)$. These two operators have the same norm, and we therefore obtain
\begin{equation}\label{eq:disc-kappa}
\kappa_h:=\|(I-\Pi_h)\mathcal L^*\|_{L^2(W)\to\widetilde H^s(\Omega)}=\|\mathcal L(I-\Pi_h)\|_{\widetilde H^s(\Omega)\to L^2(W)}\le Ch^\beta.
\end{equation}
The following lemma uses \eqref{eq:disc-kappa} to bound the discretization error for exact sampled data.

\begin{lemma}[Discretization error]\label{lem:discrete-tikhonov-error}
Under \eqref{eq:f-ass}--\eqref{eq:W-sep} and Assumptions~\ref{ass:fitted_mesh} and \ref{ass:quad_order_m}, the function $\widehat u_{0,h}^\alpha\in V_{h,0}$ denotes the minimizer of \eqref{eq:J} with the exact sampled vector $\mu_W\in Y$ in place of $\mu_W^\delta$. For every fixed $\beta\in(0,\min\{s,1/2\})$, there holds
\begin{equation}\label{eq:rho-bound}
\rho_{h,R}(\alpha):=\|\widehat u_{0,h}^\alpha-u_0^\alpha\|_{H^s(\mathbb R^d)}\le C\frac{h^\beta}{\sqrt\alpha}\|u_0\|_{H^s(\mathbb R^d)},\qquad 0<h\le1,\quad h^\beta\le\sqrt\alpha.
\end{equation}
The constant is independent of $h$, $R$, and $\alpha$.
\end{lemma}

\begin{proof}
The minimizing property of $u_0^\alpha$ in \eqref{eq:Jcont} and the identity $\mu=\mathcal Lu_0$ give $J_\alpha(u_0^\alpha)\le J_\alpha(u_0)=\alpha\|u_0\|_{H^s(\mathbb R^d)}^2$. The first-order optimality condition for \eqref{eq:Jcont} gives the continuous normal equation
\begin{equation}\label{eq:disc-continuous-normal}
(\mathcal L(u_0^\alpha-u_0),\mathcal L\phi)_{L^2(W)}+\alpha(u_0^\alpha,\phi)_{H^s(\mathbb R^d)}=0,\qquad \phi\in\widetilde H^s(\Omega).
\end{equation}
We define the residual $r_\alpha:=\mathcal L(u_0-u_0^\alpha)\in L^2(W)$. The adjoint relation rewrites \eqref{eq:disc-continuous-normal} as $(\alpha u_0^\alpha-\mathcal L^*r_\alpha,\phi)_{H^s(\mathbb R^d)}=0$ for every $\phi\in\widetilde H^s(\Omega)$. The minimizing comparison above and this identity therefore give
\begin{equation}\label{eq:disc-continuous-optimality}
\|r_\alpha\|_{L^2(W)}^2+\alpha\|u_0^\alpha\|_{H^s(\mathbb R^d)}^2\le\alpha\|u_0\|_{H^s(\mathbb R^d)}^2,\qquad \alpha u_0^\alpha=\mathcal L^*r_\alpha.
\end{equation}
Therefore the projection error $e_h:=(I-\Pi_h)u_0^\alpha$ can be estimated using \eqref{eq:disc-continuous-optimality} and \eqref{eq:disc-kappa} as follows
\[
\begin{aligned}
\|e_h\|_{H^s(\mathbb R^d)}=\alpha^{-1}\|(I-\Pi_h)\mathcal L^*r_\alpha\|_{H^s(\mathbb R^d)} \le\frac{\kappa_h}{\alpha}\|r_\alpha\|_{L^2(W)}
\le\frac{\kappa_h}{\sqrt\alpha}\|u_0\|_{H^s(\mathbb R^d)}
\le C\frac{h^\beta}{\sqrt\alpha}\|u_0\|_{H^s(\mathbb R^d)}.
\end{aligned}
\]
Let $z_h^\alpha\in V_{h,0}$ be the minimizer of $J_\alpha$ over $V_{h,0}$. Its normal equation is
\begin{equation}\label{eq:disc-galerkin-normal}
(\mathcal L(z_h^\alpha-u_0),\mathcal L\chi_h)_{L^2(W)}+\alpha(z_h^\alpha,\chi_h)_{H^s(\mathbb R^d)}=0,\qquad \chi_h\in V_{h,0}.
\end{equation}
We subtract \eqref{eq:disc-continuous-normal}, with $\phi=\chi_h$, from \eqref{eq:disc-galerkin-normal} and obtain
\[
(\mathcal L(z_h^\alpha-u_0^\alpha),\mathcal L\chi_h)_{L^2(W)}+\alpha(z_h^\alpha-u_0^\alpha,\chi_h)_{H^s(\mathbb R^d)}=0,\qquad \chi_h\in V_{h,0}.
\]
We write $z_h^\alpha-u_0^\alpha=d_h-e_h$, where $d_h:=z_h^\alpha-\Pi_hu_0^\alpha\in V_{h,0}$. The choice $\chi_h=d_h$ in the last equation consequently yields
\[
\|\mathcal Ld_h\|_{L^2(W)}^2+\alpha\|d_h\|_{H^s(\mathbb R^d)}^2=(\mathcal Le_h,\mathcal Ld_h)_{L^2(W)},
\]
where we have used the orthogonality of $\Pi_h$, which gives $(e_h,d_h)_{H^s(\mathbb R^d)}=0$. The Cauchy--Schwarz and Young inequalities imply $\alpha\|d_h\|_{H^s(\mathbb R^d)}^2\le\tfrac14\|\mathcal Le_h\|_{L^2(W)}^2$. The projection identity $\Pi_h^2=\Pi_h$ gives $\Pi_he_h=\Pi_h(I-\Pi_h)u_0^\alpha=0$, so $e_h=(I-\Pi_h)e_h$. These identities and \eqref{eq:disc-kappa} give
\[
\|d_h\|_{H^s(\mathbb R^d)}\le\frac{\|\mathcal Le_h\|_{L^2(W)}}{2\sqrt\alpha}\le\frac{\kappa_h}{2\sqrt\alpha}\|e_h\|_{H^s(\mathbb R^d)}.
\]
The identity $z_h^\alpha-u_0^\alpha=d_h-e_h$ and the bound for $e_h$ now give
\begin{equation}\label{eq:disc-galerkin-error}
\|z_h^\alpha-u_0^\alpha\|_{H^s(\mathbb R^d)}\le\left(\frac{\kappa_h}{\sqrt\alpha}+\frac{\kappa_h^2}{2\alpha}\right)\|u_0\|_{H^s(\mathbb R^d)}.
\end{equation}
The minimizing property of $z_h^\alpha$ and the contractivity of $\Pi_h$ give
\[
\alpha\|z_h^\alpha\|_{H^s(\mathbb R^d)}^2\le J_\alpha(\Pi_hu_0)=\|\mathcal L(I-\Pi_h)u_0\|_{L^2(W)}^2+\alpha\|\Pi_hu_0\|_{H^s(\mathbb R^d)}^2\le(\kappa_h^2+\alpha)\|u_0\|_{H^s(\mathbb R^d)}^2.
\]
The estimate \eqref{eq:disc-kappa} and the mesh condition $h^\beta\le\sqrt\alpha$ consequently imply
\begin{equation}\label{eq:disc-galerkin-bound}
\|z_h^\alpha\|_{H^s(\mathbb R^d)}\le C\|u_0\|_{H^s(\mathbb R^d)},
\end{equation}
where the constant \(C\) is independent of \(h\) and \(\alpha\). The normal equation \eqref{eq:normal} with exact data $\mu_W=(\mathcal Lu_0)_W$ has the variational form
\begin{equation}\label{eq:disc-exact-normal}
Q_h^W\big(\mathcal L(\widehat u_{0,h}^\alpha-u_0),\mathcal L\chi_h\big)+\alpha(\widehat u_{0,h}^\alpha,\chi_h)_{H^s(\mathbb R^d)}=0,\qquad \chi_h\in V_{h,0}.
\end{equation}
We subtract \eqref{eq:disc-galerkin-normal} from \eqref{eq:disc-exact-normal} and take $\chi_h=\widehat u_{0,h}^\alpha-z_h^\alpha$. The nonnegativity of the squared observation term and the estimates \eqref{eq:quad_error} and \eqref{eq:L_Hm_bound} give
\[
\begin{aligned}
\alpha\|\widehat u_{0,h}^\alpha-z_h^\alpha\|_{H^s(\mathbb R^d)}^2
&\le \Big|Q_h^W\big(\mathcal L(u_0-z_h^\alpha),\mathcal L(\widehat u_{0,h}^\alpha-z_h^\alpha)\big) -\big(\mathcal L(u_0-z_h^\alpha),\mathcal L(\widehat u_{0,h}^\alpha-z_h^\alpha)\big)_{L^2(W)}\Big|\\
&\le Ch^m\|u_0-z_h^\alpha\|_{H^s(\mathbb R^d)}\|\widehat u_{0,h}^\alpha-z_h^\alpha\|_{H^s(\mathbb R^d)}.
\end{aligned}
\]
This estimate and \eqref{eq:disc-galerkin-bound} imply
\begin{equation}\label{eq:disc-quadrature-error}
\begin{aligned}
\|\widehat u_{0,h}^\alpha-z_h^\alpha\|_{H^s(\mathbb R^d)}
\le C\frac{h^m}{\alpha}\|u_0-z_h^\alpha\|_{H^s(\mathbb R^d)} \le C\frac{h^m}{\alpha}\|u_0\|_{H^s(\mathbb R^d)}.
\end{aligned}
\end{equation}
The integer $m>d/2$ in Assumption \ref{ass:quad_order_m} satisfies $m\ge1>2\beta$. The conditions $h\le1$ and $h^\beta\le\sqrt\alpha$ therefore imply $h^m/\alpha\le h^{2\beta}/\alpha=(h^\beta/\sqrt\alpha)^2\le h^\beta/\sqrt\alpha$. The triangle inequality and \eqref{eq:disc-kappa}, \eqref{eq:disc-galerkin-error}, and \eqref{eq:disc-quadrature-error} prove \eqref{eq:rho-bound}. The constants are independent of $R$ because the estimates concern the fixed domains $\Omega$ and $W$ and the mesh regularity bounds are independent of $h$ and $R$.
\end{proof}

\begin{theorem}[Convergence of the reconstructed solution]\label{thm:state}
We assume the hypotheses of Lemma \ref{lem:discrete-tikhonov-error} and the noise bound \eqref{eq:noise} for the measurements $g_W^\delta\in Y$. Let the preprocessed data $\mu_W^\delta\in Y$ be defined by \eqref{eq:mu-vector}. We fix $\beta\in(0,\min\{s,1/2\})$ and choose the regularization parameter $\alpha=\alpha(h,\delta,R)>0$ such that
\begin{equation}\label{eq:alpha-rule}
\alpha\to0,\qquad \frac{(\delta+\eta(h,R)+h^\beta)^2}{\alpha}\to0
\qquad\text{as }h\to0,\ R\to\infty,\ \delta\to0.
\end{equation}
Then the reconstruction $u_h^\alpha=u_{h,f}+u_{0,h}^\alpha$, extended by zero outside $\Omega_R$, converges in $H^s(\mathbb R^d)$ to the exact solution $u=\bar u_f+u_0$ of \eqref{eq:fwd}, where $u_{0,h}^\alpha$ is given by \eqref{eq:normal}.
\end{theorem}

\begin{proof}
We subtract the exact-data variational form \eqref{eq:disc-exact-normal} from the noisy-data variational form of \eqref{eq:normal} and test with $u_{0,h}^\alpha-\widehat u_{0,h}^\alpha\in V_{h,0}$. This yields
\[
\begin{aligned}
\|L_h(u_{0,h}^\alpha-\widehat u_{0,h}^\alpha)\|_Y^2+\alpha\|u_{0,h}^\alpha-\widehat u_{0,h}^\alpha\|_{H^s(\mathbb R^d)}^2 =\langle\mu_W^\delta-\mu_W,L_h(u_{0,h}^\alpha-\widehat u_{0,h}^\alpha)\rangle_Y.
\end{aligned}
\]
The Cauchy--Schwarz and Young inequalities bound the right-hand side by $\|L_h(u_{0,h}^\alpha-\widehat u_{0,h}^\alpha)\|_Y^2+\tfrac14\|\mu_W^\delta-\mu_W\|_Y^2$. After cancellation of the observation term, \eqref{eq:delta_etabound} gives
\begin{equation}\label{eq:stab-clean}
\|u_{0,h}^\alpha-\widehat u_{0,h}^\alpha\|_{H^s(\mathbb R^d)}\le\frac{\|\mu_W^\delta-\mu_W\|_Y}{2\sqrt\alpha}\le\frac{\delta+\eta(h,R)}{2\sqrt\alpha}.
\end{equation}
The parameter rule \eqref{eq:alpha-rule} gives $h^{2\beta}/\alpha\to0$, so $h^\beta\le\sqrt\alpha$ for all sufficiently small $h$ along this limit. For the minimizer $u_0^\alpha$ of \eqref{eq:Jcont}, the triangle inequality, \eqref{eq:stab-clean}, and Lemma \ref{lem:discrete-tikhonov-error} give
\[
\|u_{0,h}^\alpha-u_0^\alpha\|_{H^s(\mathbb R^d)}\le\frac{\delta+\eta(h,R)}{2\sqrt\alpha}+C\frac{h^\beta}{\sqrt\alpha}\|u_0\|_{H^s(\mathbb R^d)}\longrightarrow0.
\]
We next show that $u_0^\alpha\to u_0$ as $\alpha\to0$. The minimizing property already used in \eqref{eq:disc-continuous-optimality} gives
\[
\|u_0^\alpha\|_{H^s(\mathbb R^d)}\le\|u_0\|_{H^s(\mathbb R^d)},\qquad \|\mathcal L(u_0^\alpha-u_0)\|_{L^2(W)}\le\sqrt\alpha\|u_0\|_{H^s(\mathbb R^d)}.
\]
The first bound ensures that every sequence $\alpha_n\to0$ has a subsequence along which $u_0^{\alpha_n}$ converges weakly to some $v\in\widetilde H^s(\Omega)$. The boundedness of $\mathcal L$ and the second bound imply $\mathcal Lv=\mathcal Lu_0$, so injectivity gives $v=u_0$. The upper bound on the norms and weak lower semicontinuity then imply convergence of the norms along this subsequence. The weak convergence is therefore strong in the Hilbert space $H^s(\mathbb R^d)$. This argument applies to every sequence $\alpha_n\to0$, and hence $u_0^\alpha\to u_0$ in $H^s(\mathbb R^d)$. The convergence established above, the approximation estimate \eqref{eq:uhf-err}, and the decompositions $u_h^\alpha=u_{0,h}^\alpha+u_{h,f}$ and $u=u_0+\bar u_f$ now give $u_h^\alpha\to u$ in $H^s(\mathbb R^d)$.
\end{proof}

Theorem \ref{thm:state} proves convergence. The next lemma gives an algebraic bound for the regularization error under a source condition (see \cite[Ch. 4]{engl1996regularization}).

The error estimate uses a source condition formulated in terms of the observation operator and its adjoint. We define
\[
A:=\mathcal L^*\mathcal L:\widetilde H^s(\Omega)\to \widetilde H^s(\Omega).
\]
Then \(A\) is bounded, self-adjoint, positive, and compact. Hence, the spectral theorem yields an orthonormal basis
\(\{e_j\}_{j=1}^\infty\subset \widetilde H^s(\Omega)\) and eigenvalues \(\lambda_j\ge 0\) such that
\(
Ae_j=\lambda_j e_j,~ j\in\mathbb N.
\)
The injectivity of $\mathcal L$ gives $\ker(A)=\{0\}$, so $\lambda_j>0$ for every $j$. For $\nu>0$, we define
\[
A^\nu v:=\sum_{j=1}^\infty \lambda_j^\nu (v,e_j)_{H^s(\mathbb R^d)}\,e_j,
\qquad v\in \widetilde H^s(\Omega),
\]
so that \(A^\nu\) is again bounded, positive, and self-adjoint.

\begin{assumption}[Source condition]\label{ass:source_state}
For $u_0=u-\bar u_f\in\widetilde H^s(\Omega)$, there exist $\nu\in(0,1]$ and $\omega,r_u\in\widetilde H^s(\Omega)$ such that
\begin{equation}\label{eq:source_state}
u_0=A^\nu\omega+r_u,\qquad \|r_u\|_{H^s(\mathbb R^d)}\le\epsilon_u,
\end{equation}
where $\epsilon_u\ge0$. The case $r_u=0$ is the exact source condition.
\end{assumption}
Since $A^\nu$ is self-adjoint and injective, its range is dense in $\widetilde H^s(\Omega)$. Thus \eqref{eq:source_state} holds for every $\epsilon_u>0$, although $\|\omega\|_{H^s(\mathbb R^d)}$ may grow as $\epsilon_u\to0$. For fixed $\nu$, bounded source elements along a sequence with $\epsilon_u\to0$ admit a subsequence converging weakly to some $\omega_\ast$. Taking weak limits in \eqref{eq:source_state} gives $u_0=A^\nu\omega_\ast$. The exact condition is equivalent to $\sum_{j=1}^\infty\lambda_j^{-2\nu}|(u_0,e_j)_{H^s(\mathbb R^d)}|^2<\infty$.

\begin{lemma}\label{lem:tikh_bias}
Suppose that the solution $u$ of \eqref{eq:fwd} has the decomposition \eqref{eq:u-decomp} and that $u_0$ satisfies Assumption \ref{ass:source_state}. The minimizer $u_0^\alpha$ of \eqref{eq:Jcont} then satisfies
\begin{equation}\label{eq:bias_bound}
\|u_0^\alpha-u_0\|_{H^s(\mathbb R^d)}
\le \alpha^\nu\|\omega\|_{H^s(\mathbb R^d)}+\epsilon_u.
\end{equation}
In particular, if $u_0\in\operatorname{Range}(A^\nu)$, then one may take $r_u=0$ and obtain
\[
\|u_0^\alpha-u_0\|_{H^s(\mathbb R^d)}
\le \alpha^\nu\|\omega\|_{H^s(\mathbb R^d)}.
\]
\end{lemma}

\begin{proof}
For every $\varphi\in \widetilde H^s(\Omega)$, the minimizer $u_0^\alpha$ satisfies the Euler--Lagrange equation
\[
(\mathcal Lu_0^\alpha-\mu,\mathcal L\varphi)_{L^2(W)} + \alpha (u_0^\alpha,\varphi)_{ H^s(\R^d)} = 0.
\]
By the definition of the adjoint, this is equivalent to
\[
(\mathcal L^*(\mathcal Lu_0^\alpha-\mu)+\alpha u_0^\alpha,\varphi)_{ H^s(\R^d)} = 0
\qquad \forall \varphi\in \widetilde H^s(\Omega).
\]
Hence
\begin{equation}\label{eq:matrixA_u0}
(\mathcal L^*\mathcal L+\alpha I)u_0^\alpha = \mathcal L^*\mu.
\end{equation}
Since \(\mu=\mathcal Lu_0\), we obtain \( (A+\alpha I)u_0^\alpha=Au_0\). Moreover, for every \(v\in \widetilde H^s(\Omega)\), we have
\[
((A+\alpha I)v,v)_{H^s(\mathbb R^d)}
=
\|\mathcal Lv\|_{L^2(W)}^2+\alpha\|v\|_{H^s(\mathbb R^d)}^2
\ge \alpha\|v\|_{H^s(\mathbb R^d)}^2,
\]
so \(A+\alpha I\) is invertible and bounded. Using Assumption \ref{ass:source_state}, this yields
\[
u_0-u_0^\alpha = \alpha(A+\alpha I)^{-1}u_0 = \alpha(A+\alpha I)^{-1}A^\nu\omega + \alpha(A+\alpha I)^{-1}r_u.
\]
The expansions $\omega=\sum_j\omega_j e_j$ and $r_u=\sum_j r_{u,j}e_j$, with $\omega_j:=(\omega,e_j)_{H^s(\mathbb R^d)}$ and $r_{u,j}:=(r_u,e_j)_{H^s(\mathbb R^d)}$, give
\[
\alpha(A+\alpha I)^{-1}A^\nu\omega = \sum_j \frac{\alpha\lambda_j^\nu}{\lambda_j+\alpha}\,\omega_j e_j,
\qquad
\alpha(A+\alpha I)^{-1}r_u =
\sum_j \frac{\alpha}{\lambda_j+\alpha}\,r_{u,j}e_j.
\]
Therefore,
\[
\|\alpha(A+\alpha I)^{-1}A^\nu\omega\|_{H^s(\mathbb R^d)}^2
=
\sum_j \Bigl(\frac{\alpha\lambda_j^\nu}{\lambda_j+\alpha}\Bigr)^2 |\omega_j|^2
\le \alpha^{2\nu}\|\omega\|_{H^s(\mathbb R^d)}^2,
\]
since \(t^\nu\le 1+t\) for all \(t\ge 0\) when \(0<\nu\le 1\). Also,
\[
\|\alpha(A+\alpha I)^{-1}r_u\|_{H^s(\mathbb R^d)}^2
=
\sum_j \Bigl(\frac{\alpha}{\lambda_j+\alpha}\Bigr)^2 |r_{u,j}|^2
\le \|r_u\|_{H^s(\mathbb R^d)}^2
\le \epsilon_u^2.
\]
Combining the two bounds yields \eqref{eq:bias_bound}. The final statement is the special case \(r_u=0\).
\end{proof}

\begin{theorem}[Error estimate for the reconstructed solution]\label{thm:state_rate}
We assume the hypotheses of Lemma \ref{lem:discrete-tikhonov-error} and the noise bound \eqref{eq:noise}, with the preprocessed noisy data $\mu_W^\delta$ defined by \eqref{eq:mu-vector}. The exact solution is $u=\bar u_f+u_0\in H^s(\mathbb R^d)$, where $u_0\in\widetilde H^s(\Omega)$ satisfies Assumption \ref{ass:source_state} for some $\omega\in\widetilde H^s(\Omega)$. For a fixed $\beta\in(0,\min\{s,1/2\})$, $0<h\le1$, and $h^\beta\le\sqrt\alpha$, the reconstruction $u_h^\alpha\in H^s(\mathbb R^d)$ satisfies
\begin{equation}\label{eq:state_error_est}
\begin{aligned}
\|u_h^\alpha-u\|_{H^s(\mathbb R^d)}\le C\Big(\frac{\delta+\eta(h,R)+h^\beta\|u_0\|_{H^s(\mathbb R^d)}}{\sqrt\alpha}+\alpha^\nu\|\omega\|_{H^s(\mathbb R^d)}+\epsilon_u+h^{p+1-s}\|\bar u_f\|_{W^{p+1,\infty}(\mathbb R^d)}\Big).
\end{aligned}
\end{equation}
The constant $C$ is independent of $h$, $R$, $\delta$, and $\alpha$. For a fixed $\nu\in(0,1]$, we allow the source representations in \eqref{eq:source_state} to depend on $h$, $R$, and $\delta$, and assume that $\|\omega\|_{H^s(\mathbb R^d)}$ remains bounded independently of these parameters. The conditions
\begin{gather}
h^\beta=\mathcal O(\delta+\eta(h,R)),\label{ass:discre}\\
c_1(\delta+\eta(h,R))^{\frac{2}{2\nu+1}}\le\alpha\le c_2(\delta+\eta(h,R))^{\frac{2}{2\nu+1}},\label{eq:alpha_apriori_rate}\\
\epsilon_u=\mathcal O\left((\delta+\eta(h,R))^{\frac{2\nu}{2\nu+1}}\right),\label{eq:source-residual-rate}
\end{gather}
with fixed $0<c_1\le c_2$, imply
\begin{equation}\label{est:rate}
\|u_h^\alpha-u\|_{H^s(\mathbb R^d)}=\mathcal O\left((\delta+\eta(h,R))^{\frac{2\nu}{2\nu+1}}\right)\qquad\text{as }\delta+\eta(h,R)\to0^+.
\end{equation}
\end{theorem}

\begin{proof}
The decomposition
\begin{equation}\label{eq:decomu_alpha}
u_h^\alpha-u=(u_{0,h}^\alpha-\widehat u_{0,h}^\alpha)+(\widehat u_{0,h}^\alpha-u_0^\alpha)+(u_0^\alpha-u_0)+(u_{h,f}-\bar u_f)
\end{equation}
is an identity in $H^s(\mathbb R^d)$. The four terms are bounded by \eqref{eq:stab-clean}, Lemma \ref{lem:discrete-tikhonov-error}, Lemma \ref{lem:tikh_bias}, and \eqref{eq:uhf-err}, respectively. This proves \eqref{eq:state_error_est}. The conditions \eqref{ass:discre} and \eqref{eq:alpha_apriori_rate} give
\[
\frac{\delta+\eta(h,R)+h^\beta}{\sqrt\alpha}+\alpha^\nu=\mathcal O\left((\delta+\eta(h,R))^{\frac{2\nu}{2\nu+1}}\right).
\]
This implies $h^\beta/\sqrt\alpha\to0$, so the mesh condition holds for sufficiently small $\delta+\eta(h,R)>0$. The inequality $p+1-s>\beta$ also gives $h^{p+1-s}<h^\beta=\mathcal O(\delta+\eta(h,R))$. The uniform bound on $\|\omega\|_{H^s(\mathbb R^d)}$ and \eqref{eq:source-residual-rate} therefore give \eqref{est:rate} from \eqref{eq:state_error_est}.
\end{proof}

We next obtain a logarithmic error estimate without Assumption \ref{ass:source_state}. The proof follows the argument for the regularization error in \cite{DwivediRailoRuppConductivity2026}, which combines the forward equation with the quantitative unique continuation estimate in \cite[Prop. 6.1]{ghosh2020uniqueness}.

\begin{lemma}[Logarithmic solution error]\label{lem:log-solution-error}
We assume the hypotheses of Lemma \ref{lem:discrete-tikhonov-error} and that the exact potential $q\in L^\infty(\Omega)$ satisfies Assumption \ref{ass:nonresonance}. Let the exact solution of \eqref{eq:fwd} be $u=\bar u_f+u_0\in H^s(\mathbb R^d)$ and $u_0^\alpha$ be the continuous minimizer defined in \eqref{eq:u0alpha}. Then there exist constants $C>0$ and $\sigma>0$ such that
\begin{equation}\label{eq:log-bias}
\|u_0^\alpha-u_0\|_{H^s(\mathbb R^d)}\le C[\log(2+\alpha^{-1})]^{-\sigma}
\end{equation}
for sufficiently small $\alpha>0$. We also assume the noise bound \eqref{eq:noise} and use the preprocessed data $\mu_W^\delta$ defined in \eqref{eq:mu-vector}. For a fixed $\beta\in(0,\min\{s,1/2\})$, we set $D:=\delta+\eta(h,R)+h^\beta$ and choose $\alpha=D$. The corresponding reconstruction $u_h^\alpha$ then satisfies
\begin{equation}\label{eq:log-solution-error}
\|u_h^\alpha-u\|_{H^s(\mathbb R^d)}\le C[\log(2+D^{-1})]^{-\sigma}
\end{equation}
for sufficiently small $D>0$. The constants are independent of $h$, $R$, $\delta$, and $\alpha$.
\end{lemma}
\begin{proof}
Since the exact solution $u\in H^s(\mathbb R^d)$ of \eqref{eq:fwd} has the decomposition $u=\bar u_f+u_0$, where $u_0\in\widetilde H^s(\Omega)$ and $\bar u_f\in C_c^\infty(\mathbb R^d)$, the weak equation \eqref{eq:weak_forward} yields
\[
(u_0,v)_{H^s(\mathbb R^d)}=(F_0,v)_{L^2(\Omega)},\qquad v\in\widetilde H^s(\Omega),
\]
where $F_0:=(1-q)u_0-((-\Delta)^s\bar u_f)|_\Omega\in L^2(\Omega)$. The choice $\phi=u_0^\alpha-u_0\in\widetilde H^s(\Omega)$ in the continuous normal equation \eqref{eq:disc-continuous-normal}, together with the identity above, implies
\[
\|\mathcal L(u_0^\alpha-u_0)\|_{L^2(W)}^2+\alpha\|u_0^\alpha-u_0\|_{H^s(\mathbb R^d)}^2=-\alpha(F_0,u_0^\alpha-u_0)_{L^2(\Omega)}.
\]
The nonnegativity of the first term and the Cauchy--Schwarz inequality give, after division by $\alpha$,
\begin{equation}\label{eq:log-energy-L2}
\|u_0^\alpha-u_0\|_{H^s(\mathbb R^d)}^2\le\|F_0\|_{L^2(\Omega)}\|u_0^\alpha-u_0\|_{L^2(\Omega)}.
\end{equation}
The estimate \eqref{eq:log-energy-L2} reduces the proof of \eqref{eq:log-bias} to an estimate of the $L^2(\Omega)$ error. The minimizing comparison in \eqref{eq:disc-continuous-optimality} gives, for $\alpha>0$,
\begin{equation}
\begin{aligned}
\|u_0^\alpha-u_0\|_{H^s(\mathbb R^d)}\le2\|u_0\|_{H^s(\mathbb R^d)},\label{eq:log-state-bound}\qquad\qquad
\|\mathcal L(u_0^\alpha-u_0)\|_{L^2(W)}&\le\sqrt\alpha\|u_0\|_{H^s(\mathbb R^d)}.
\end{aligned}
\end{equation}
The estimate in \cite[Prop. 6.1]{ghosh2020uniqueness}, stated for bounded Lipschitz domains, extends to the set $W$ by applying it on a fixed ball compactly contained in $W$. The test functions on that ball form a subspace of $\widetilde H^s(W)$ with the same $H^s(\mathbb R^d)$ norm, so restriction does not increase the $H^{-s}$ norm. The embedding $L^2(W)\hookrightarrow H^{-s}(W)$ and the second bound in \eqref{eq:log-state-bound} give
\begin{equation}\label{eq:log-exterior-bound}
\|\mathcal L(u_0^\alpha-u_0)\|_{H^{-s}(W)}\le C\|\mathcal L(u_0^\alpha-u_0)\|_{L^2(W)}\le C\sqrt\alpha\|u_0\|_{H^s(\mathbb R^d)}.
\end{equation}
If $\|\mathcal L(u_0^\alpha-u_0)\|_{H^{-s}(W)}=0$ then $u_0^\alpha=u_0$ by unique continuation, since $u_0^\alpha-u_0$ vanishes on $W$. In the remaining case, we use \cite[Prop. 6.1]{ghosh2020uniqueness}, with its Sobolev index $s'=0$. The first bound in \eqref{eq:log-state-bound} and \eqref{eq:log-exterior-bound} then yield
\begin{equation}\label{eq:log-L2-error}
\|u_0^\alpha-u_0\|_{L^2(\Omega)}\le C\left[\log\left(\frac{c}{\sqrt\alpha}\right)\right]^{-2\sigma}\le C[\log(2+\alpha^{-1})]^{-2\sigma}
\end{equation}
for sufficiently small $\alpha>0$, where $C,c>0$ are independent of $\alpha$ and $2\sigma$ denotes the positive exponent. The estimates \eqref{eq:log-energy-L2} and \eqref{eq:log-L2-error} then prove \eqref{eq:log-bias}. We now combine \eqref{eq:log-bias} with the data, discretization, and extension errors. The decomposition \eqref{eq:decomu_alpha}, together with \eqref{eq:stab-clean}, \eqref{eq:rho-bound}, \eqref{eq:uhf-err}, and \eqref{eq:log-bias}, gives
\[
\|u_h^\alpha-u\|_{H^s(\mathbb R^d)}\le C\left(\frac{D}{\sqrt\alpha}+[\log(2+\alpha^{-1})]^{-\sigma}+h^{p+1-s}\right)
\]
whenever $h^\beta\le\sqrt\alpha$. The choice $\alpha=D<1$ satisfies this mesh condition because $h^\beta\le D\le\sqrt D$. The inequalities $p+1-s>\beta$ and $h<1$ also give $h^{p+1-s} < h^\beta\le D$. The error bound above therefore becomes
\[
\|u_h^\alpha-u\|_{H^s(\mathbb R^d)}\le C\left(\sqrt D+[\log(2+D^{-1})]^{-\sigma}+D\right).
\]
The powers $\sqrt D$ and $D$ are bounded above by $C[\log(2+D^{-1})]^{-\sigma}$ for sufficiently small $D>0$, which proves \eqref{eq:log-solution-error}.
\end{proof}

\subsection{Reconstruction of the potential \texorpdfstring{\(q\)}{q}}\label{subsec:qrec}
The weak equation \eqref{eq:weak_forward} for the exact solution $u\in H^s(\mathbb R^d)$ and potential $q\in L^\infty(\Omega)$ gives $((-\Delta)^s u)|_\Omega=-qu\in L^2(\Omega)$. We denote the zero extension of this function to $\mathbb R^d$ by $w$. For $\alpha_w>0$, the approximation $w_h^\alpha\in V_{h,0}$ is defined by
\begin{equation}\label{eq:w-def}
(w_h^\alpha,\varphi_h)_{L^2(\Omega)}+\alpha_w a_R(w_h^\alpha,\varphi_h)=a_R(u_h^\alpha,\varphi_h),\qquad \varphi_h\in V_{h,0}.
\end{equation}
The dependence on $\alpha_w$ is omitted from the notation. For the basis $\{\varphi_i\}_{i=1}^{N_0}$ of $V_{h,0}$, the coefficient vectors $\mathbf w^\alpha,\mathbf v_h^\alpha\in\mathbb R^{N_0}$ of $w_h^\alpha$ and $u_{0,h}^\alpha$, respectively, satisfy
\[
(M_0+\alpha_w A_0)\mathbf w^\alpha=A_0\mathbf v_h^\alpha+\mathbf b^{(\mathrm{ext})},\qquad b_i^{(\mathrm{ext})}:=a_R(u_{h,f},\varphi_i),\quad i=1,\dots,N_0.
\]
Here $A_0$ and $M_0$ are defined in \eqref{eq:A0M0}. The matrix $M_0+\alpha_w A_0$ is symmetric positive definite. 

\begin{theorem}[Error in the fractional Laplacian]\label{thm:fractional-error}
We assume the mesh conditions in Section \ref{subsec:trunc-fe}, including Assumption \ref{ass:fitted_mesh}. The exact solution $u\in H^s(\mathbb R^d)$ satisfies \eqref{eq:weak_forward} for $q\in L^\infty(\Omega)$, and $w\in L^2(\mathbb R^d)$ is the zero extension of $-qu$ from $\Omega$. Then, for the reconstruction $u_h^\alpha$ and the approximation $w_h^\alpha\in V_{h,0}$ defined by \eqref{eq:w-def} with $\alpha_w>0$, the errors \( \mathcal E_u:=\|u_h^\alpha-u\|_{H^s(\mathbb R^d)}\) and  \(\mathcal E_w:=\|w_h^\alpha-w\|_{L^2(\Omega)}\) satisfy
\begin{equation}\label{eq:fractional-error-general}
\mathcal E_w\le \frac{\mathcal E_u}{2\sqrt{\alpha_w}}+\inf_{\xi_h\in V_{h,0}}\left(\|w-\xi_h\|_{L^2(\Omega)}^2+\alpha_w a(\xi_h,\xi_h)\right)^{1/2}.
\end{equation}
If the zero extension $w$ also satisfies $w\in H^t(\mathbb R^d)$ and $\operatorname{supp}w\subset\subset\Omega$ for some $0<t\le s$, then, for sufficiently small $h$ and $\alpha_w$, there holds
\begin{equation}\label{eq:fractional-error}
\mathcal E_w\le \frac{\mathcal E_u}{2\sqrt{\alpha_w}}+C\left(h^t+\alpha_w^{t/(2s)}\right)\|w\|_{H^t(\mathbb R^d)}.
\end{equation}
The constant $C$ is independent of $h$, $R$, $\delta$, $\alpha$, and $\alpha_w$. 
\end{theorem}

\begin{proof}
Let us consider $z_h\in V_{h,0}$, defined by
\begin{equation}\label{eq:fractional-comparison}
(z_h,\varphi_h)_{L^2(\Omega)}+\alpha_w a(z_h,\varphi_h)=(w,\varphi_h)_{L^2(\Omega)},\qquad \varphi_h\in V_{h,0}.
\end{equation}
Since \eqref{eq:weak_forward} gives $a(u,\varphi_h)=(w,\varphi_h)_{L^2(\Omega)}$ and  $a_R=a$ for the zero-extended discrete functions, we subtract \eqref{eq:fractional-comparison} from \eqref{eq:w-def} and take $\varphi_h=w_h^\alpha-z_h\in V_{h,0}$. The Cauchy--Schwarz and Young inequalities yield
\[
\begin{aligned}
\|w_h^\alpha-z_h\|_{L^2(\Omega)}^2+\alpha_w a(w_h^\alpha-z_h,w_h^\alpha-z_h)
=a(u_h^\alpha-u,w_h^\alpha-z_h)\le\frac{\mathcal E_u^2}{4\alpha_w}+\alpha_w a(w_h^\alpha-z_h,w_h^\alpha-z_h),
\end{aligned}
\]
which implies
\begin{equation}\label{eq:fractional-data-error}
\|w_h^\alpha-z_h\|_{L^2(\Omega)}\le\frac{\mathcal E_u}{2\sqrt{\alpha_w}}.
\end{equation}
The equation \eqref{eq:fractional-comparison} is also the optimality condition for minimizing the functional $\|v_h-w\|_{L^2(\Omega)}^2+\alpha_w a(v_h,v_h)$ over $v_h\in V_{h,0}$. The minimizing property of $z_h$ therefore gives
\[
\|z_h-w\|_{L^2(\Omega)}^2\le\|z_h-w\|_{L^2(\Omega)}^2+\alpha_w a(z_h,z_h)\le\|w-\xi_h\|_{L^2(\Omega)}^2+\alpha_w a(\xi_h,\xi_h),\qquad \forall\xi_h\in V_{h,0}.
\]
The triangle inequality and \eqref{eq:fractional-data-error}, followed by the infimum over $\xi_h\in V_{h,0}$, prove \eqref{eq:fractional-error-general}.

For \eqref{eq:fractional-error}, we choose $\ell=\max\{h,\alpha_w^{1/(2s)}\}$ and $w_\ell=\rho_\ell*w$, where $\rho_\ell(x)=\ell^{-d}\rho(x/\ell)$ for a fixed nonnegative $\rho\in C_c^\infty(B_1(0))$ with $\int_{\mathbb R^d}\rho=1$. The compact support assumption gives $w_\ell\in C_c^\infty(\Omega)$ for sufficiently small $\ell$, and the standard mollification estimates yield
\begin{equation}\label{eq:w-mollification}
\|w-w_\ell\|_{L^2(\mathbb R^d)}\le C\ell^t\|w\|_{H^t(\mathbb R^d)},\qquad
\|w_\ell\|_{H^r(\mathbb R^d)}\le C\ell^{t-r}\|w\|_{H^t(\mathbb R^d)},\quad r\in\{s,1\}.
\end{equation}
We choose $\xi_h$ as the Scott--Zhang quasi-interpolant of $w_\ell$ in the continuous piecewise affine space with zero boundary values on $\partial\Omega$, which is contained in $V_{h,0}$ because $p\ge1$. The interpolation estimates in \cite[Sec. 4.1]{BorthagarayLiNochetto2024} give
\begin{equation}\label{eq:w-interpolation}
\|w_\ell-\xi_h\|_{L^2(\Omega)}+h^s\|w_\ell-\xi_h\|_{H^s(\mathbb R^d)}\le Ch\|w_\ell\|_{H^1(\mathbb R^d)}.
\end{equation}
The triangle inequality, \eqref{eq:w-mollification}, \eqref{eq:w-interpolation}, and $h\le\ell$ imply
\[
\begin{aligned}
\|w-\xi_h\|_{L^2(\Omega)}
&\le C\left(\ell^t+h\ell^{t-1}\right)\|w\|_{H^t(\mathbb R^d)}
\le C\ell^t\|w\|_{H^t(\mathbb R^d)},\\
a(\xi_h,\xi_h)^{1/2}
&\le C\left(\ell^{t-s}+h^{1-s}\ell^{t-1}\right)\|w\|_{H^t(\mathbb R^d)}
\le C\ell^{t-s}\|w\|_{H^t(\mathbb R^d)}.
\end{aligned}
\]
The inequality $\sqrt{\alpha_w}\le\ell^s$ therefore gives
\[
\left(\|w-\xi_h\|_{L^2(\Omega)}^2+\alpha_w a(\xi_h,\xi_h)\right)^{1/2}\le C\ell^t\|w\|_{H^t(\mathbb R^d)}.
\]
We use this choice of $\xi_h$ in \eqref{eq:fractional-error-general} and the inequality $\ell^t\le h^t+\alpha_w^{t/(2s)}$ to obtain \eqref{eq:fractional-error}. The constants are independent of $R$ because the interpolation estimates use the mesh on the fixed domain $\Omega$ with uniform mesh regularity constants.
\end{proof}

The estimate \eqref{eq:fractional-error} requires additional regularity and compact support of $w$ inside $\Omega$. A sufficient condition is $q\in C^{0,s}(\overline\Omega)$ with $\operatorname{supp}q\subset\subset\Omega$, which gives $w\in H^t(\mathbb R^d)$ for every $0<t<s$ and $\operatorname{supp}w\subset\subset\Omega$. The product estimate for the Gagliardo seminorm gives
\[
\|w\|_{H^t(\mathbb R^d)}\le C_t\|q\|_{C^{0,s}(\overline\Omega)}\|u\|_{H^s(\mathbb R^d)},\qquad 0<t<s.
\]
The same estimate holds at $t=s$ if $q$ is Lipschitz continuous, with $\|q\|_{C^{0,1}(\overline\Omega)}$ in place of $\|q\|_{C^{0,s}(\overline\Omega)}$. For each fixed admissible $t$, we assume the hypotheses and parameter conditions yielding \eqref{est:rate} in Theorem \ref{thm:state_rate} and choose $\alpha_w=(\delta+\eta(h,R))^{\frac{4s\nu}{(2\nu+1)(s+t)}}$. The estimate \eqref{eq:fractional-error} and the mesh condition \eqref{ass:discre} then imply
\begin{equation}\label{eq:fractional-rate}
\mathcal E_w=\mathcal O\left((\delta+\eta(h,R))^{\frac{2\nu t}{(2\nu+1)(s+t)}}\right)\qquad\text{as }\delta+\eta(h,R)\to0^+.
\end{equation}

We now use the reconstructed forward solution $u_h^\alpha$ and its fractional Laplacian approximation $w_h^\alpha$ to recover $q$ on a bounded Lipschitz set $\Omega'$ with $\overline{\Omega'}\subset\Omega$. The restriction to \(\Omega'\) is used to obtain $u\ge\kappa>0$ in $\Omega'$ under the positivity assumptions introduced below. The identity $w=-qu$ from \eqref{eq:fwd} gives
\begin{equation}\label{eq:q-formal}
q=-\frac{w}{u}=-\frac{(-\Delta)^s u}{u}\qquad\text{a.e. on }\{x\in\Omega'\mid u(x)\ne0\}.
\end{equation} 
The discrete quotient $-w_h^\alpha/u_h^\alpha$ corresponding to \eqref{eq:q-formal} can amplify errors where $|u_h^\alpha|$ is small, so we use \eqref{eq:q-min}. Finite element error estimates for coefficient identification in an elliptic equation with a power nonlinearity are given in \cite{ChenLinYousept2026}. The semilinear extension in \cite{dwivedi2026numerical} recovers the coefficient of a known power of the solution from one exterior measurement.

The coefficient mesh $\mathcal T_h'$ of $\Omega'$ is shape regular and quasi-uniform, with mesh size at most $h$. For a fixed degree $r\ge0$, the coefficient space is
\[
Q_h=\{p_h\in L^2(\Omega')\mid p_h|_K\in\mathbb P_r(K)\text{ for every element }K\in\mathcal T_h'\},
\]
where $\mathbb P_r(K)$ denotes the space of polynomials of total degree at most $r$ on $K$. Let the functions $\{\psi_j\}_{j=1}^{N_q}$ be a basis of $Q_h$. We seek $q_h^\alpha\in Q_h$ in the form
\[
q_h^\alpha=\sum_{j=1}^{N_q} q_j\psi_j,
\qquad \mathbf q=(q_1,\dots,q_{N_q})^\top\in\mathbb R^{N_q}.
\]

For the reconstructed solution $u_h^\alpha=u_{0,h}^\alpha+u_{h,f}\in V_h$ and the fractional approximation $w_h^\alpha\in V_{h,0}$ from \eqref{eq:w-def}, the coefficient $q_h^\alpha\in Q_h$ is defined, for $\alpha_q>0$, by
\begin{equation}\label{eq:q-min}
q_h^\alpha
:= \operatorname*{argmin}_{q_h\in Q_h}
\Big(
\|w_h^\alpha+q_hu_h^\alpha\|_{L^2(\Omega')}^2
+\alpha_q\|q_h\|_{L^2(\Omega')}^2
\Big).
\end{equation}

\begin{theorem}[Existence and uniqueness of the discrete coefficient reconstruction]\label{thm:qmin}
Let $u_h^\alpha\in V_h$ and $w_h^\alpha\in V_{h,0}$ be given (recovered) quantities.
Then, for every fixed $\alpha_q>0$, the minimization problem \eqref{eq:q-min} has a unique minimizer
$q_h^\alpha\in Q_h$. Moreover, $q_h^\alpha$ is characterized by the variational normal equations
\begin{equation}\label{eq:q-var-disc-alpha-final}
\int_{\Omega'} \big((u_h^\alpha)^2+\alpha_q\big)\,q_h^\alpha\,\phi_h\,dx
= -\int_{\Omega'} w_h^\alpha u_h^\alpha\,\phi_h\,dx
\qquad \forall \phi_h\in Q_h,
\end{equation}
and the coefficient vector $\mathbf q$ is the unique solution of the SPD linear system
\begin{subequations}
\begin{equation}\label{eq:q-system}
G_\alpha \mathbf q = \mathbf b,
\end{equation}
where
\begin{equation}\label{eq:Geps}
(G_\alpha)_{ij}:= \int_{\Omega'} \big((u_h^\alpha)^2+\alpha_q\big)\psi_i\psi_j\,dx, \qquad
b_i := -\int_{\Omega'} w_h^\alpha u_h^\alpha \psi_i\,dx.
\end{equation}
\end{subequations}
\end{theorem}

\begin{proof}
Define
\[
F(q_h):=\|w_h^\alpha+q_h\,u_h^\alpha\|_{L^2(\Omega')}^2+\alpha_q\|q_h\|_{L^2(\Omega')}^2,
\qquad q_h\in Q_h.
\]
Since $Q_h$ is finite-dimensional and $u_h^\alpha\in V_h$ and $w_h^\alpha\in V_{h,0}$ are fixed, the mapping $q_h\mapsto F(q_h)$
is a continuous quadratic functional on $Q_h$.
Coercivity follows from the regularization term
\begin{equation}\label{eqn:coercive}
F(q_h)\ \ge\ \alpha_q\|q_h\|_{L^2(\Omega')}^2
\qquad\forall q_h\in Q_h.
\end{equation}
Hence $F(q_h)\to\infty$ as $\|q_h\|_{L^2(\Omega')}\to\infty$ in $Q_h$.
For any $q_h,r_h\in Q_h$,
\[
F(q_h)-F(r_h)-\langle F'(r_h),q_h-r_h\rangle
= \| (q_h-r_h)u_h^\alpha\|_{L^2(\Omega')}^2+\alpha_q\|q_h-r_h\|_{L^2(\Omega')}^2
\ \ge\ \alpha_q\|q_h-r_h\|_{L^2(\Omega')}^2,
\]
which is strictly positive whenever $q_h\neq r_h$ since $\alpha_q>0$. Thus $F$ is strictly convex.

Because $Q_h$ is finite-dimensional, coercivity \eqref{eqn:coercive} implies that any minimizing sequence for $F$
is bounded in $L^2(\Omega')$, hence has a convergent subsequence in $Q_h$.
By continuity of $F$, the limit of this subsequence is a minimizer.
Uniqueness follows from strict convexity: if two minimizers existed, their midpoint would yield a strictly smaller value.

Let $q_h^{\alpha}$ be the unique minimizer.
For any direction $\phi_h\in Q_h$ and $t\in\mathbb{R}$ define $\Phi(t):=F(q_h^{\alpha}+t\phi_h)$.
Since $F$ is quadratic, $\Phi$ is a differentiable convex function and $t=0$ is its minimizer, so $\Phi'(0)=0$.
Compute
\[
\Phi'(0)
= 2\int_{\Omega'} (w_h^\alpha+q_h^{\alpha}u_h^\alpha)\, (u_h^\alpha\phi_h)\,dx
+2\alpha_q\int_{\Omega'} q_h^{\alpha}\phi_h\,dx.
\]
Thus, this gives
\[
\int_{\Omega'} (u_h^\alpha)^2\, q_h^{\alpha}\phi_h\,dx
+\alpha_q\int_{\Omega'} q_h^{\alpha}\phi_h\,dx
= -\int_{\Omega'} w_h^\alpha u_h^\alpha \phi_h\,dx,
\]
which is exactly \eqref{eq:q-var-disc-alpha-final}.

Write $q_h^{\alpha}=\sum_{j=1}^{N_q} q_j\psi_j$ and choose $\phi_h=\psi_i$ in \eqref{eq:q-var-disc-alpha-final}. This gives the system \eqref{eq:q-system}.
To prove $G_\alpha$ is SPD, let $\mathbf z =(z_1,\dots,z_{N_q})^\top\in\mathbb{R}^{N_q}\setminus\{0\}$ and set $z_h=\sum_{j=1}^{N_q} z_j\psi_j\in Q_h$.
Then
\[
\mathbf z^\top G_\alpha \mathbf z
= \int_{\Omega'} (u_h^\alpha)^2 z_h^2\,dx + \alpha_q\int_{\Omega'} z_h^2\,dx
\ge \alpha_q\|z_h\|_{L^2(\Omega')}^2>0,
\]
since $\alpha_q>0$ and $z_h\not\equiv0$ whenever $\mathbf z\neq0$. Hence $G_\alpha$ is SPD.
\end{proof}
For a basis of $Q_h$ whose functions are supported on individual elements, the matrix $G_\alpha$ is block diagonal and SPD. The assembly of $G_\alpha$ and $\mathbf b$ requires only local elementwise quadrature of the integrals in \eqref{eq:Geps}.

\subsection{Error analysis for the approximation of the potential \texorpdfstring{\(q\)}{q}}\label{subsec:qconv_alpha_final}

\begin{lemma}[Bounds for the exact solution \(u\)]\label{lem:positive-solution}
Suppose that $0\le q\le M$ a.e. in $\Omega$ and that $f\in C_c^\infty(W)$ is nonnegative and nonzero. Then the solution $u\in H^s(\mathbb R^d)$ of \eqref{eq:fwd} satisfies
\[
0\le u(x)\le\|f\|_{L^\infty(W)}\quad\text{for almost every }x\in\mathbb R^d,\qquad u(x)\ge\kappa\quad\text{for almost every }x\in\Omega',
\]
where $\kappa>0$ depends only on $M$, $f$, $\Omega$, $\Omega'$, $d$, and $s$.
\end{lemma}
\begin{proof}
The maximum principle and the estimate in \cite[Props. 3.1 and 3.3]{LaiLin2019}, applied to \eqref{eq:fwd}, give $0\le u\le\|f\|_{L^\infty(W)}$ a.e. in $\mathbb R^d$. Moreover, the decomposition $u=u_0+\bar u_f$, where $u_0\in\widetilde H^s(\Omega)$ and $\bar u_f\in C_c^\infty(\mathbb R^d)$ for $f\in C_c^\infty(W)$, and the weak formulation \eqref{eq:weak_forward} give
\begin{equation}\label{eq:comp1}
((-\Delta)^s+q)u_0=F_f\quad\text{in }\Omega,\qquad F_f(x):=c_{d,s}\int_W\frac{f(y)}{|x-y|^{d+2s}}\,dy,\quad x\in\Omega.
\end{equation}
The assumptions $f\ge0$ and $f\not\equiv0$ imply $\inf_\Omega F_f>0$. We choose a fixed $\chi\in C_c^\infty(\Omega)$ with $0\le\chi\le1$ and $\chi=1$ on $\overline{\Omega'}$, and define
\[
\kappa:=\frac{\inf_\Omega F_f}{1+\|(-\Delta)^s\chi\|_{L^\infty(\Omega)}+M}>0.
\]
The inequalities $q\le M$ and $\chi\le1$, together with this choice of $\kappa$, give
\begin{equation}\label{eq:comp2}
\begin{aligned}
((-\Delta)^s+q)(\kappa\chi) \le \kappa(\|(-\Delta)^s\chi\|_{L^\infty(\Omega)}+M)
\le\inf_\Omega F_f\le F_f, \qquad \text{ a.e. in }\Omega. 
\end{aligned}
\end{equation}
Both $u_0$ and $\kappa\chi$ vanish outside $\Omega$, so the comparison principle in \cite[Cor. 3.2]{LaiLin2019} for \eqref{eq:comp1} and \eqref{eq:comp2} implies $u_0\ge\kappa\chi$. The identities $u=u_0$ in $\Omega$ and $\chi=1$ on $\Omega'$ therefore imply $u\ge\kappa$ a.e. in $\Omega'$.
\end{proof}

We denote the elementwise $L^2$ projection of $q$ by $q_I\in Q_h$. The local polynomial inverse inequality and the $L^2$ contractivity of the projection give $\|q_I\|_{L^\infty(\Omega')}\le C_r\|q\|_{L^\infty(\Omega')}$, where $C_r$ is independent of $h$. For $q\in H^t(\Omega')$ with $0<t<1$, the element mean $q_K:=|K|^{-1}\int_Kq(x)\,dx$ is an admissible comparison since $r\ge0$. The best approximation property and the fractional Poincar\'e inequality give
\[
\begin{aligned}
\|q-q_I\|_{L^2(K)}^2 \le\|q-q_K\|_{L^2(K)}^2
=\frac1{2|K|}\int_K\int_K|q(x)-q(y)|^2\,dx\,dy\le Ch^{2t}|q|_{H^t(K)}^2.
\end{aligned}
\]
For $q\in H^1(\Omega')$, the usual Poincaré inequality gives $\|q-q_I\|_{L^2(K)}\le Ch\|\nabla q\|_{L^2(K)}$. We sum the squared local estimates over the elements and take square roots to obtain
\begin{equation}\label{eq:Qapprox}
\|q-q_I\|_{L^2(\Omega')}\le Ch^t\|q\|_{H^t(\Omega')},\qquad 0<t\le1,
\end{equation}
where $C$ is independent of $h$. The estimate \eqref{eq:Qapprox}, the density of smooth functions in $L^2(\Omega')$, and the contractivity of the projection imply $\|q-q_I\|_{L^2(\Omega')}\to0$ as $h\to0$ for every $q\in L^2(\Omega')$.

\begin{theorem}[Coefficient error]\label{thm:q_disc_error}
Let $q\in L^\infty(\Omega)$ and the exact solution $u\in H^s(\mathbb R^d)$ of \eqref{eq:fwd} satisfy the hypotheses of Lemma \ref{lem:positive-solution}. We also assume the mesh conditions of Theorem \ref{thm:fractional-error}. Then, for $1\ge \alpha_q>0$, the minimizer $q_h^\alpha\in Q_h$ of \eqref{eq:q-min} satisfies
\begin{equation}\label{eq:q_disc_err_final}
\|q_h^\alpha-q\|_{L^2(\Omega')}\le C\left(\frac{\mathcal E_u+\mathcal E_w+\|q-q_I\|_{L^2(\Omega')}}{\sqrt{\alpha_q}}+\sqrt{\alpha_q}\right).
\end{equation}
The constant $C$ is independent of $h$, $R$, $\delta$, $\alpha$, $\alpha_w$, and $\alpha_q$. In particular, $q_h^\alpha\to q$ in $L^2(\Omega')$ if $\alpha_q\to0$ and $(\mathcal E_u+\mathcal E_w+\|q-q_I\|_{L^2(\Omega')})/\sqrt{\alpha_q}\to0$.
\end{theorem}
\begin{proof}
Taking $\phi_h = q_h^\alpha-q_I\in Q_h$ in \eqref{eq:q-var-disc-alpha-final} implies
\begin{equation}\label{eq:id1}
\begin{aligned}
&\|u_h^\alpha(q_h^\alpha-q_I)\|_{L^2(\Omega')}^2+\alpha_q\|q_h^\alpha-q_I\|_{L^2(\Omega')}^2\\&\qquad\qquad\qquad =-(w_h^\alpha+q_Iu_h^\alpha,u_h^\alpha(q_h^\alpha-q_I))_{L^2(\Omega')}-\alpha_q(q_I,q_h^\alpha-q_I)_{L^2(\Omega')}.
\end{aligned}
\end{equation}
Since Lemma \ref{lem:positive-solution} provides $u\ge\kappa>0$ in $\Omega'$, $\|q/u\|_{L^\infty(\Omega')}\le M/\kappa$. Thus the projection orthogonality $(q_I-q,q_h^\alpha-q_I)_{L^2(\Omega')}=0$ and the identity $q=(q/u)u_h^\alpha+(q/u)(u-u_h^\alpha)$ yield
\[
\begin{aligned}
(q_I,q_h^\alpha-q_I)_{L^2(\Omega')} =\left(\frac{q}{u},u_h^\alpha(q_h^\alpha-q_I)\right)_{L^2(\Omega')}+\left(\frac{q}{u},(u-u_h^\alpha)(q_h^\alpha-q_I)\right)_{L^2(\Omega')}.
\end{aligned}
\]
We substitute this identity into \eqref{eq:id1} and use $\|u-u_h^\alpha\|_{L^2(\Omega')}\le\mathcal E_u$. The Cauchy--Schwarz and Young's inequalities imply
\[
\alpha_q\|q_h^\alpha-q_I\|_{L^2(\Omega')}^2
\le C\left(\|w_h^\alpha+q_Iu_h^\alpha\|_{L^2(\Omega')}^2+\alpha_q^2+\alpha_q\mathcal E_u^2\right).
\]
Thus,
\begin{equation}\label{id:eq2}
\|q_h^\alpha-q_I\|_{L^2(\Omega')}
\le C\left(\frac{\|w_h^\alpha+q_Iu_h^\alpha\|_{L^2(\Omega')}}{\sqrt{\alpha_q}}+\sqrt{\alpha_q}+\mathcal E_u\right).
\end{equation}
Since $w=-qu$ in $\Omega$, we have \( w_h^\alpha+q_Iu_h^\alpha=(w_h^\alpha-w)+q_I(u_h^\alpha-u)+(q_I-q)u\) in \(\Omega'\). Therefore the bounds $\|q_I\|_{L^\infty(\Omega')}\le C_rM$ and $\|u\|_{L^\infty(\Omega')}\le\|f\|_{L^\infty(W)}$ imply
\[
\begin{aligned}
\|w_h^\alpha+q_Iu_h^\alpha\|_{L^2(\Omega')}
&\le\mathcal E_w+C_rM\mathcal E_u+\|f\|_{L^\infty(W)}\|q-q_I\|_{L^2(\Omega')}.
\end{aligned}
\]
The triangle inequality, \eqref{id:eq2}, and the preceding bound give \eqref{eq:q_disc_err_final}, since $1\le\alpha_q^{-1/2}$ for $0<\alpha_q\le1$. The convergence assertion follows from this estimate.
\end{proof}

We now assume that the zero extension of $q$ belongs to $H^t(\mathbb R^d)$ for some $0<t\le s$ and retain the positivity assumptions above. These conditions allow us to combine the logarithmic estimate \eqref{eq:log-solution-error} for the forward solution \(u\) with the fractional Laplacian error estimate \eqref{eq:fractional-error} and the coefficient error estimate \eqref{eq:q_disc_err_final}. R\"uland \cite{Ruland2021stability} proved a related single-measurement stability estimate for potentials with a common $C^{0,s}$ bound. 

\begin{theorem}[Logarithmic coefficient error]\label{thm:stability2d}
We assume the hypotheses of Lemma \ref{lem:log-solution-error} and Theorem \ref{thm:q_disc_error}. We assume further that $\operatorname{supp}q\subset\overline{\Omega'}$ and that the zero extension of $q$ belongs to $H^t(\mathbb R^d)$ for a fixed $0<t\le s$. We recall $D=\delta+\eta(h,R)+h^\beta$, with fixed $\beta\in(0,\min\{s,1/2\})$, and the exponent $\sigma>0$ from Lemma \ref{lem:log-solution-error}. We choose the regularization parameters as
\begin{equation}\label{eq:log-parameter-choice}
\alpha=D,\qquad \alpha_w=[\log(2+D^{-1})]^{-2s\sigma/(s+t)},\qquad \alpha_q=[\log(2+D^{-1})]^{-\sigma t/(s+t)}.
\end{equation}
Then, for sufficiently small $D>0$, the reconstruction $q_h^\alpha\in Q_h$ satisfies
\begin{equation}\label{eq:stability}
\|q_h^\alpha-q\|_{L^2(\Omega')}\le C[\log(2+D^{-1})]^{-\sigma t/(2(s+t))}.
\end{equation}
The constant $C$ may depend on $\|q\|_{H^t(\mathbb R^d)}$ and the fixed problem parameters, but is independent of $h$, $R$, and $\delta$. For fixed $W$ and $f$, \eqref{eqn:est_tail_term} gives $\eta(h,R)\le C(R-\rho_f)^{-2s}$ uniformly in $h$. Thus $D\to0$ as $h\to0$, $R\to\infty$, and $\delta\to0$, and \eqref{eq:stability} gives $q_h^\alpha\to q$ in $L^2(\Omega')$.
\end{theorem}
\begin{proof}
Since $q\in L^\infty(\Omega)$ has a zero extension in $H^t(\mathbb R^d)$ and $0<t\le s$, the bounds for $u$ in Lemma \ref{lem:positive-solution} imply $w=-qu\in H^t(\mathbb R^d)$, with
\[
\|w\|_{H^t(\mathbb R^d)}\le C\left(\|q\|_{L^\infty(\Omega)}\|u\|_{H^s(\mathbb R^d)}+\|u\|_{L^\infty(\mathbb R^d)}\|q\|_{H^t(\mathbb R^d)}\right),
\]
where we used the product estimate in the Gagliardo seminorm and the embedding $H^s(\mathbb R^d)\hookrightarrow H^t(\mathbb R^d)$. The condition \(\operatorname{supp}q\subset\overline{\Omega'}\) also gives $\operatorname{supp}w\subset\overline{\Omega'}\subset\subset\Omega$. Thus Theorem \ref{thm:fractional-error} with $\mathcal E_u\le C[\log(2+D^{-1})]^{-\sigma}$ from Lemma \ref{lem:log-solution-error} and the choice of $\alpha_w$ in \eqref{eq:log-parameter-choice} yield
\[
\begin{aligned}
\mathcal E_w \le C\left(\frac{[\log(2+D^{-1})]^{-\sigma}}{\sqrt{\alpha_w}}+h^t+\alpha_w^{t/(2s)}\right) \le C\left([\log(2+D^{-1})]^{-\sigma t/(s+t)}+h^t\right).
\end{aligned}
\]
Since $h^\beta\le D$, we have $h^t\le D^{t/\beta}\le C[\log(2+D^{-1})]^{-\sigma t/(s+t)}$ for sufficiently small $D>0$. These bounds, estimate \eqref{eq:Qapprox}, and $t/(s+t)<1$ therefore imply
\[
\mathcal E_u+\mathcal E_w+\|q-q_I\|_{L^2(\Omega')}\le C[\log(2+D^{-1})]^{-\sigma t/(s+t)}.
\]
The coefficient error estimate \eqref{eq:q_disc_err_final} and the choice of $\alpha_q$ in \eqref{eq:log-parameter-choice} now give
\[
\begin{aligned}
\|q_h^\alpha-q\|_{L^2(\Omega')} \le C\left(\frac{[\log(2+D^{-1})]^{-\sigma t/(s+t)}}{\sqrt{\alpha_q}}+\sqrt{\alpha_q}\right) \le C[\log(2+D^{-1})]^{-\sigma t/(2(s+t))},
\end{aligned}
\]
which proves \eqref{eq:stability}.
\end{proof}

\section{Numerical experiments}\label{sec:numerics}
We use $p=p_W=1$ and solve \eqref{eq:normal} by a weighted singular value decomposition. The reference forward problems use $4000$ graded intervals in one dimension and a uniform triangulation of $(-1,1)^2$ with $640$ subdivisions per coordinate in two dimensions. The reference solution gives the preprocessed data $\mu_{W}\in Y$. For relative noise level $\delta$, we set
\[
\mu_W^\delta=\mu_W+\delta\,\|\mu_W\|_Y\,\frac{\xi}{\|\xi\|_Y},
\]
where $\xi\in Y$ is a standard Gaussian vector and $\delta$ denotes the relative noise level in this section. The absolute noise bound is $\delta_{\mathrm{abs}}=\delta\|\mu_W\|_Y$. The noise is added to $\mu_W$ directly, so the term $\eta(h,R)$ does not occur. Theorem \ref{thm:state} requires $\alpha\to0$ and $(\delta_{\mathrm{abs}}+\eta(h,R)+h^\beta)^2/\alpha\to0$ as $h\to0$, $R\to\infty$, and $\delta_{\mathrm{abs}}\to0$, for a fixed $\beta\in(0,\min\{s,1/2\})$. The experiments below examine the dependence on noise on fixed meshes.
The MATLAB code for reproducing the numerical experiments is available in the \href{https://github.com/AndreasRupp/galerkin-finite-elements-for-fractional-calderon}{project repository}.

The superscripts $\alpha,\delta$ indicate the state regularization parameter and the relative noise level. We suppress the dependence on $\alpha_w$ and $\alpha_q$.

\subsection{Example 4.1 (1D, compactly supported Lipschitz potential).}\label{subsec:ex41}
Let $\Omega=(-1,1)$ and
\[
W=(-3,-1-d_0)\cup(1+d_0,3),\qquad d_0=0.05.
\]
We take $s=0.6$ and the compactly supported bump
\[
q(x)=10\,(0.75-x^2)_+,\qquad (t)_+=\max\{t,0\},
\quad \Omega'=\bigl(-\sqrt{0.75},\sqrt{0.75}\bigr).
\]
The exterior datum is
\[
f(x)=\zeta\Bigl(\frac{x+1.3}{0.2}\Bigr)+\zeta\Bigl(\frac{x-1.3}{0.2}\Bigr),\qquad
\zeta(t)=
\begin{cases}
\exp\bigl(1-(1-t^2)^{-1}\bigr), & |t|<1,\\
0, & |t|\ge1.
\end{cases}
\]
The meshes of $\Omega$ and each interval of $W$ are uniform with element diameters $h_\Omega=0.0025$ and $h_W=0.0195$. The space $Q_h$ has degree $r=1$ on a uniform mesh of $\Omega'$ with element diameter approximately $2.5\times10^{-3}$. We evaluate \eqref{eq:Geps} by six-point Gauss quadrature on the intersections of the two interior meshes. The representative reconstructions use the relative noise levels \(
\delta\in\{10^{-7},10^{-5},10^{-3},10^{-1}\}\) and the parameters
\[
\alpha=\max\{\delta^{3/2},10^{-12}\},\qquad \alpha_w=\delta^{2/3},\qquad \alpha_q=\max\{10^{-2}\delta,10^{-12}\}.
\]
The lower bound for $\alpha$ is active for $\delta<10^{-8}$.

Figure \ref{fig:ex41_fourpanel_summary}(A) shows the recovered interior component $u_{0,h}^{\alpha,\delta}$. Panel (B) shows the total error, the error for noiseless data, and the error due to noise in $H^s(\mathbb R)$. Panel (C) shows the recovered potential. Panel (D) shows the mean $L^2(\Omega')$ error over $10$ realizations at each of $30$ logarithmically spaced levels $\delta\in[10^{-10},10^{-1}]$. The dashed curve is $C|\log\delta|^{-\gamma}$ with $C=50$ and $\gamma\approx2.13$. The parameters are fitted by least squares to the logarithms of the mean errors, and $C$ is rounded to one significant digit for the plot.

\begin{figure}[t]
  \centering
  \begin{subfigure}[t]{0.49\textwidth}
    \centering
    \includegraphics[width=\textwidth]{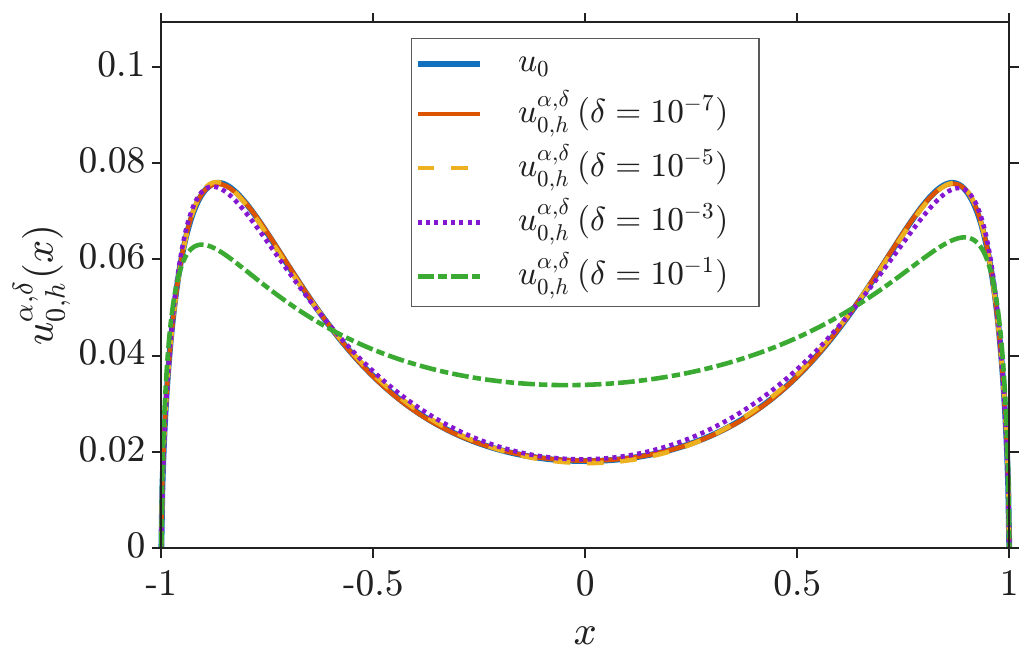}
    \caption{Reconstruction of $u_{0}$.}
    \label{fig:ex41_fourpanel_u}
  \end{subfigure}\hfill
  \begin{subfigure}[t]{0.49\textwidth}
    \centering
    \includegraphics[width=\textwidth]{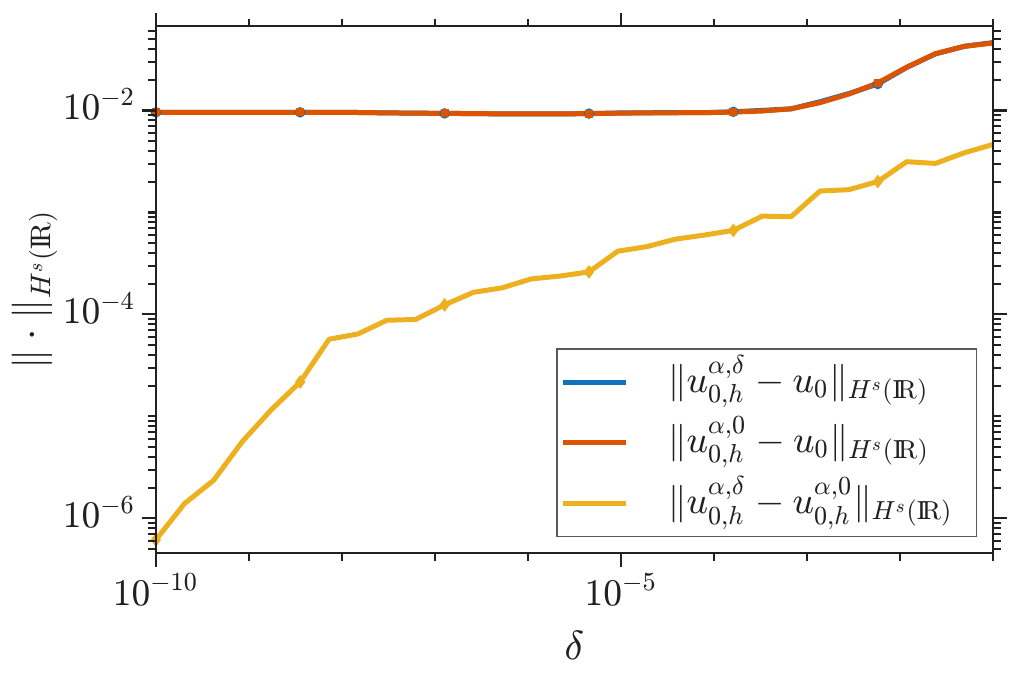}
    \caption{Noise decomposition for the error in $u_{0,h}^{\alpha,\delta}$.}
    \label{fig:ex41_fourpanel_decomp}
  \end{subfigure}

  \vspace{0.5em}

  \begin{subfigure}[t]{0.49\textwidth}
    \centering
    \includegraphics[width=\textwidth]{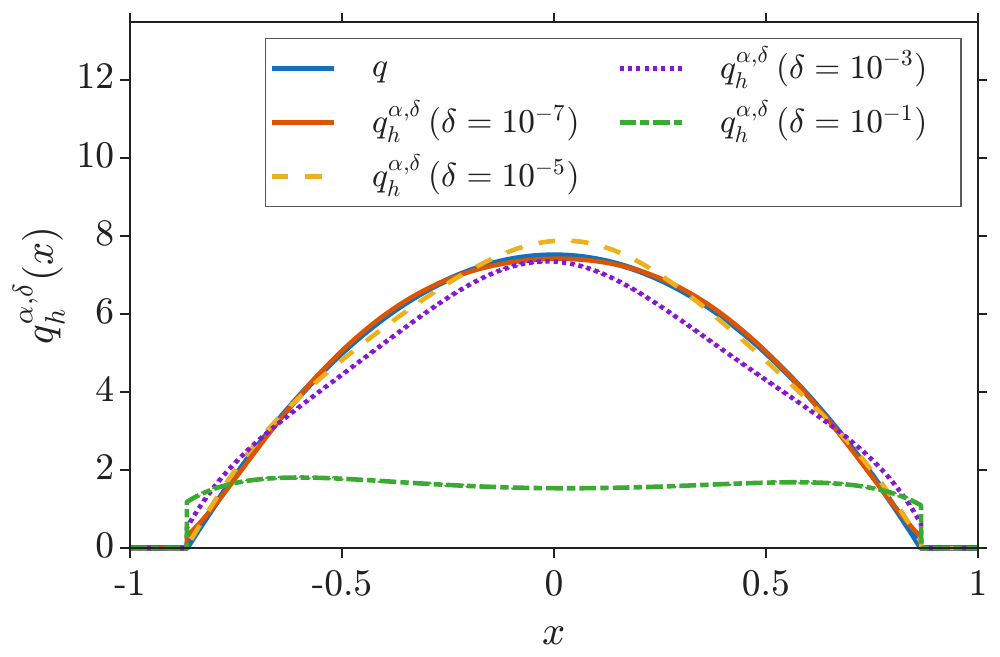}
    \caption{Reconstruction of the coefficient $q$.}
    \label{fig:ex41_fourpanel_q}
  \end{subfigure}\hfill
  \begin{subfigure}[t]{0.49\textwidth}
    \centering
    \includegraphics[width=\textwidth]{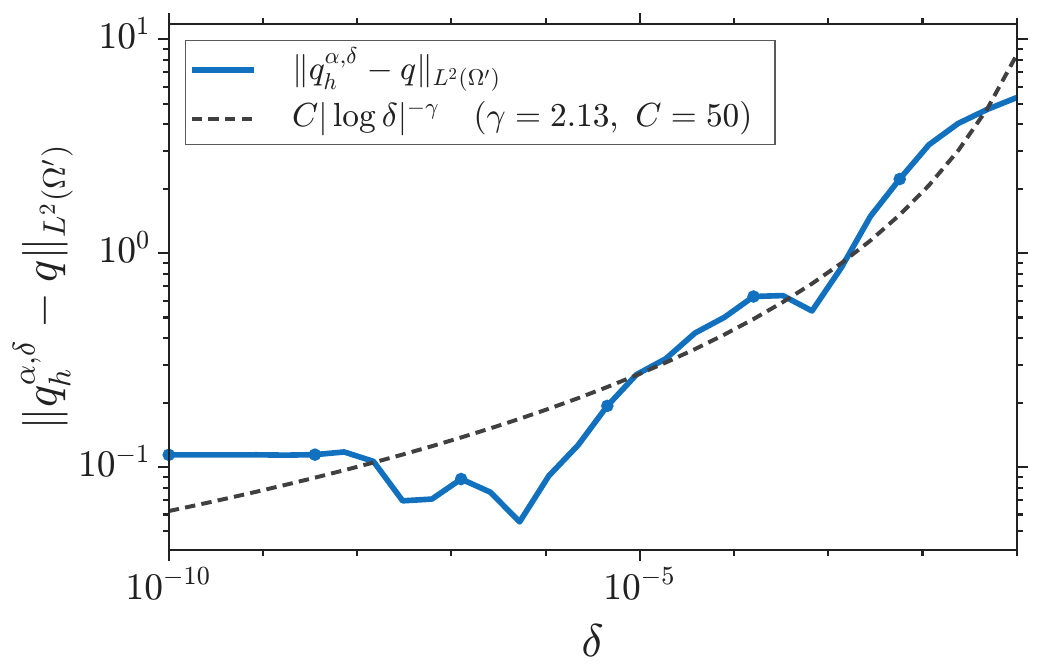}
    \caption{Potential error plot with logarithmic $x$-axis. }
    \label{fig:ex41_fourpanel_stab}
  \end{subfigure}
  \caption{Example \ref{subsec:ex41}. Panel (A) shows the recovered component $u_{0,h}^{\alpha,\delta}$ for the representative noise levels. Panel (B) shows the total error, the error for noiseless data, and the error due to noise. Panel (C) shows the recovered coefficient $q_h^{\alpha,\delta}$, and panel (D) shows the mean potential error $\|q_h^{\alpha,\delta}-q\|_{L^2(\Omega')}$ as a function of the relative noise level $\delta$.}
  \label{fig:ex41_fourpanel_summary}
\end{figure}

\subsection{Example 4.2 (1D, discontinuous potential).}\label{subsec:ex42}
We use the domain, observation set, exterior datum, and fractional order $s=0.6$ of Example \ref{subsec:ex41}. The true potential in $\Omega = (-1,1)$ is
\begin{equation}\label{eq:ex43_qtrue}
q(x)=
\begin{cases}
1, & x\in(-\tfrac12,\tfrac12),\\[2pt]
0, & x\in\Omega\setminus(-\tfrac12,\tfrac12).
\end{cases}
\end{equation}
The data are perturbed with the relative noise levels $\delta\in\{10^{-8},10^{-7},10^{-6},10^{-5}\}$. The coefficient is reconstructed on $\Omega'=(-0.8,0.8)$ and extended by $0$ outside $\Omega'$ for visualization. The coefficient mesh consists of the $640$ elements of the mesh of $\Omega$ that lie in $\Omega'$, so the two jumps of $q$ are element boundaries. The zero extension of $q$ belongs to $H^t(\mathbb R)$ for every $t<1/2$.

For each $\delta$, we first compute the approximation $u_{0,h}^{\alpha,\delta}$ of $u_0$ using the Tikhonov unique continuation step from Section \ref{sec:recon}.
The corresponding reconstructions of $q$ are shown in Figure \ref{fig:ex43_q}(A). As $\delta$ increases, the reconstruction of $u_0$ deteriorates, and this degradation propagates to the coefficient reconstruction.

We compare two coefficient reconstructions in the space $Q_h$ with polynomial degree $r=0$. The baseline $q_{h,\mathrm{L2}}^{\alpha,\delta}\in Q_h$ minimizes \eqref{eq:q-min}. We also consider the TV-regularized reconstruction defined by
\begin{equation}\label{eq:ex42_qtv}
q_{h,\mathrm{TV}}^{\alpha,\delta}\in
\arg\min_{q_h\in Q_h}
\Bigl(
\|w_h^{\alpha,\delta}+q_hu_h^{\alpha,\delta}\|_{L^2(\Omega')}^2
+\alpha_q\|q_h\|_{L^2(\Omega')}^2
+\alpha_{\mathrm{tv}}\sum_{i=1}^{N_q-1}|q_{i+1}-q_i|
\Bigr),
\end{equation}
where $q_i$ is the value of $q_h$ on the $i$-th cell, ordered from left to right. The sum is the total variation of the piecewise constant function $q_h$.
The TV term in \eqref{eq:ex42_qtv} penalizes each jump in proportion to its magnitude and permits sharp interfaces \cite{rudin1992tv,strong2003tv,acar1994bv}. The parameters $\alpha$, $\alpha_w$, and $\alpha_q$ are those of Example \ref{subsec:ex41}. We choose the TV parameter from the cell values $q_i$ of the baseline reconstruction in \eqref{eq:q-min} by setting
\[
\sigma_q:=\operatorname{median}_{1\le i\le N_q-1}|q_{i+1}-q_i|+10^{-14},\qquad \alpha_{\mathrm{tv}}=10^{-2}\sigma_q.
\]
This rule uses reconstructed quantities only. The minimization \eqref{eq:ex42_qtv} is solved by ADMM \cite{boyd2011admm}. The iteration stops when the primal and dual residuals satisfy the absolute tolerance $10^{-9}$ and the relative tolerance $10^{-7}$ and the relative duality gap is below $10^{-8}$, with at most $10^4$ iterations.

The final reconstruction assumes that the potential equals a constant in $[0,1]$ on one interval and vanishes elsewhere. The detected interval extends from the first to the last cell where the minimizer of \eqref{eq:ex42_qtv} exceeds $0.5$. We minimize \eqref{eq:q-min} over this single constant on the detected interval. In Figure \ref{fig:ex43_q}, the reconstructed interfaces lie within $0.02$ of the true jumps, and the recovered constant lies between $0.89$ and $0.91$.

\begin{figure}[t]
\centering
\begin{subfigure}{0.49\textwidth}
  \centering
  \includegraphics[width=\linewidth]{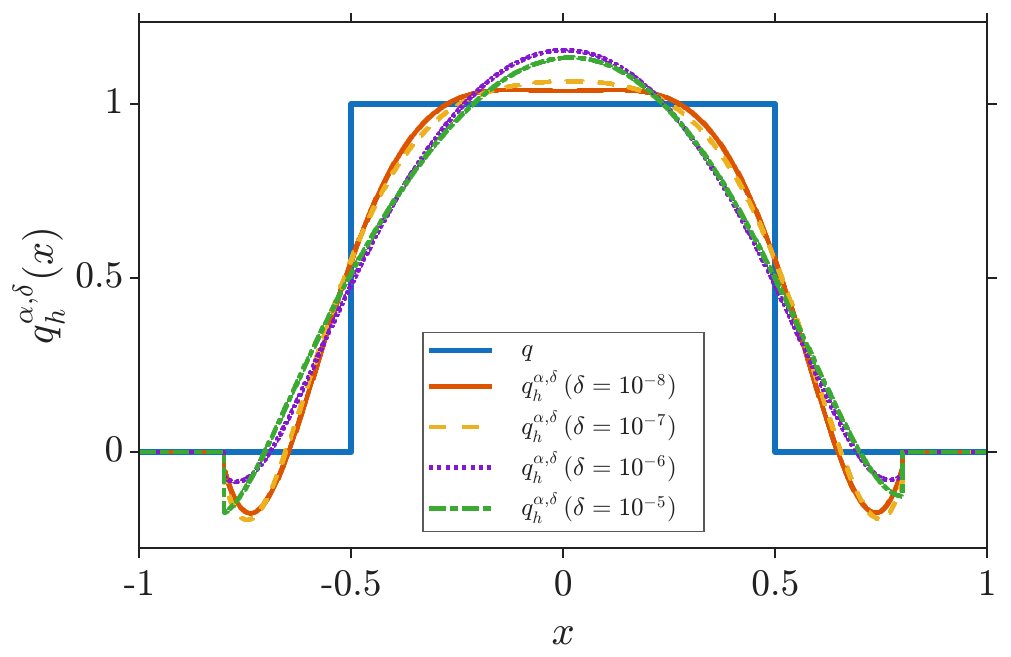}
  \caption{Baseline quadratic reconstruction \eqref{eq:q-min}.}
  \label{fig:ex43_q_l2}
\end{subfigure}\hfill
\begin{subfigure}{0.49\textwidth}
  \centering
    \includegraphics[width=\linewidth]{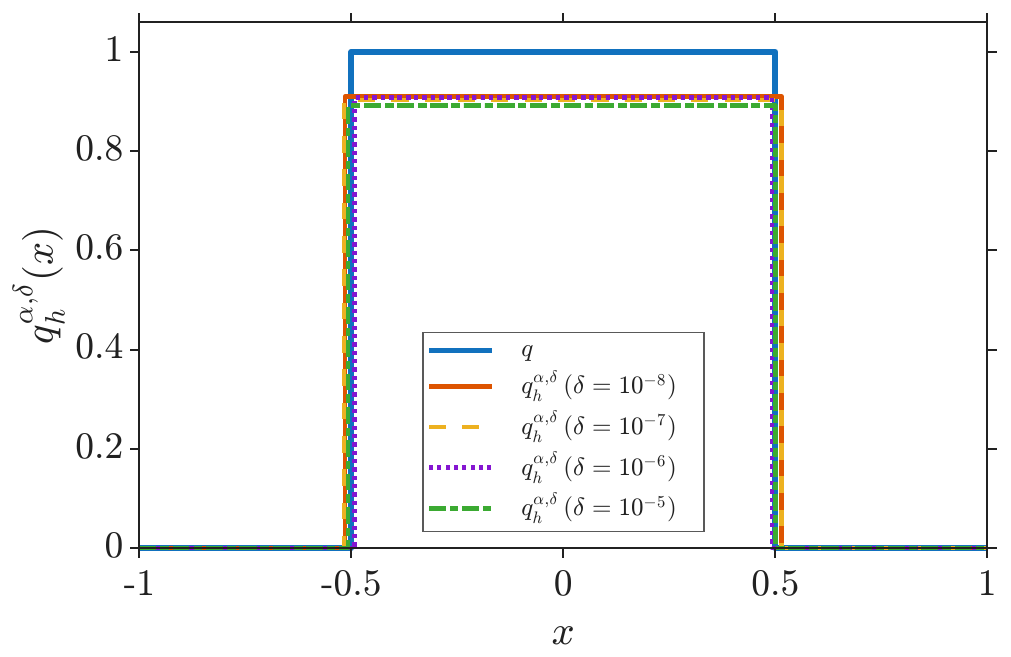}
 \caption{TV Reconstruction \eqref{eq:ex42_qtv}.}
  \label{fig:ex43_q_tv}
\end{subfigure}
\caption{Example \ref{subsec:ex42}: discontinuous potential reconstructions on $\Omega'$.}
\label{fig:ex43_q}
\end{figure}

\subsection{Example 4.3 (2D, compactly supported bump potential).}
\label{sec:ex43-2d}

We consider the two-dimensional geometry used in \cite[Ex. 4.4]{li2025numerical}.
Let
\[
\Omega = (-1,1)^2,\qquad \Omega_R = (-3,3)^2,\qquad \Omega' = (-0.75,0.75)^2 \subset\subset \Omega.
\]
We define the exterior measurement region with separation distance $d_0>0$ by
\[
W=\Omega_R\setminus[-1-d_0,1+d_0]^2,\qquad d_0=0.05.
\]
The true coefficient is the compactly supported bump
\[
q(x,y)=100\,((0.75)^2-x^2)_+^{\,3}\,((0.75)^2-y^2)_+^{\,3}.
\]
The exterior datum is $f(x)=e^{-1}\sum_{c}\zeta(|x-c|/0.45)$ with $\zeta$ from Example \ref{subsec:ex41}, where $c$ runs over the points $(\pm1.6,0)$ and $(0,\pm1.6)$. We fix the fractional order $s=0.5$ and reconstruct on $\Omega'$ with degree $r=1$, extending the reconstruction by zero outside $\Omega'$ for visualization on $\Omega$. The parameters are $\alpha=\max\{0.1\,\delta^{3/2},2\cdot10^{-14}\}$, $\alpha_w=\delta^{2/3}$, and $\alpha_q=\max\{10^{-2}\delta,10^{-13}\}$. For the stability test, we use $20$ logarithmically spaced values $\delta \in [10^{-10},10^{-1}]$ and $15$ noise realizations for each value. The interior and coefficient meshes consist of triangles with diameter $h_\Omega=h_Q=\sqrt{2}/40$, and the observation mesh has maximum element diameter $h_W=13\sqrt{2}/80$.  Figure \ref{fig:ex43-2d-u} shows the exact (reference) and recovered forward solution at $\delta=10^{-8}$. In Figure \ref{fig:ex43-2d-q}(C), the mean $H^s(\mathbb R^2)$ error due to noise increases with $\delta$ and remains below the error for noiseless data.

\begin{figure}[t]
\centering
\begin{subfigure}[t]{0.48\textwidth}
  \centering
  \includegraphics[width=\linewidth]{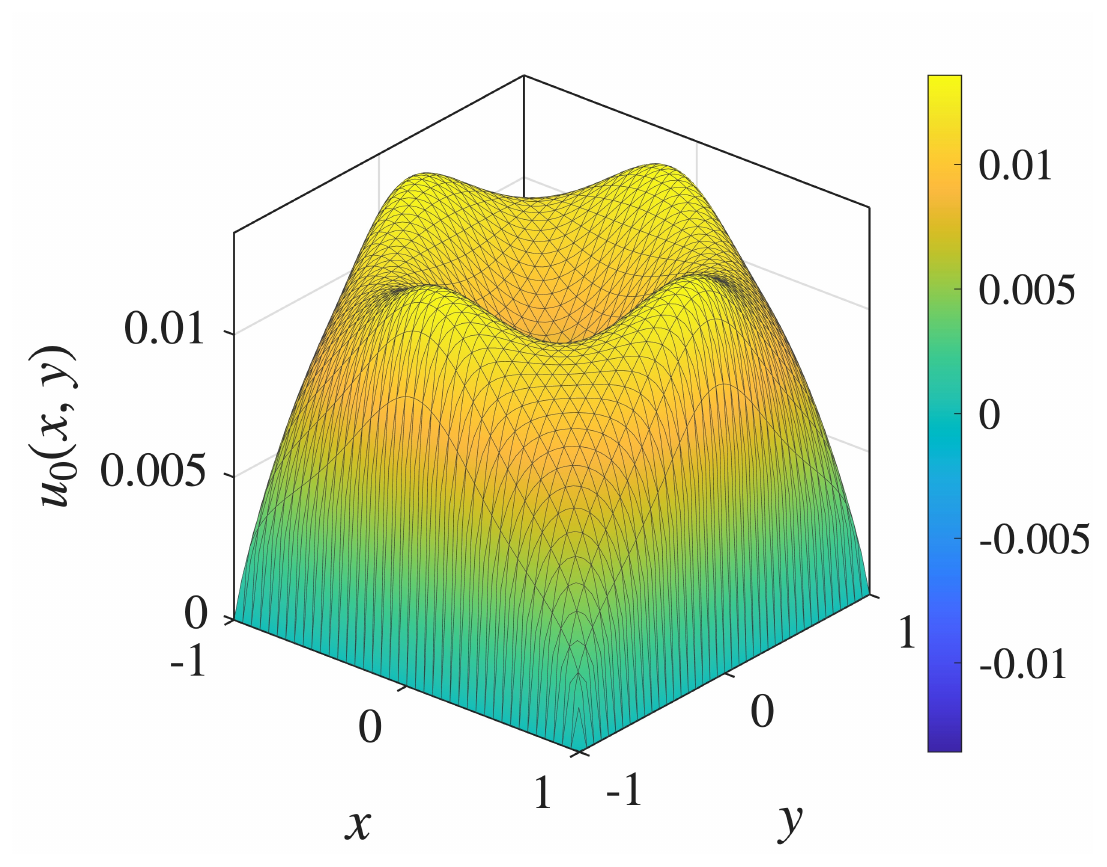}
       \caption{$u_0$ in $\Omega$.}
  \label{fig:ex43-2d-utrue}
\end{subfigure}\hfill
\begin{subfigure}[t]{0.48\textwidth}
  \centering
  \includegraphics[width=\linewidth]{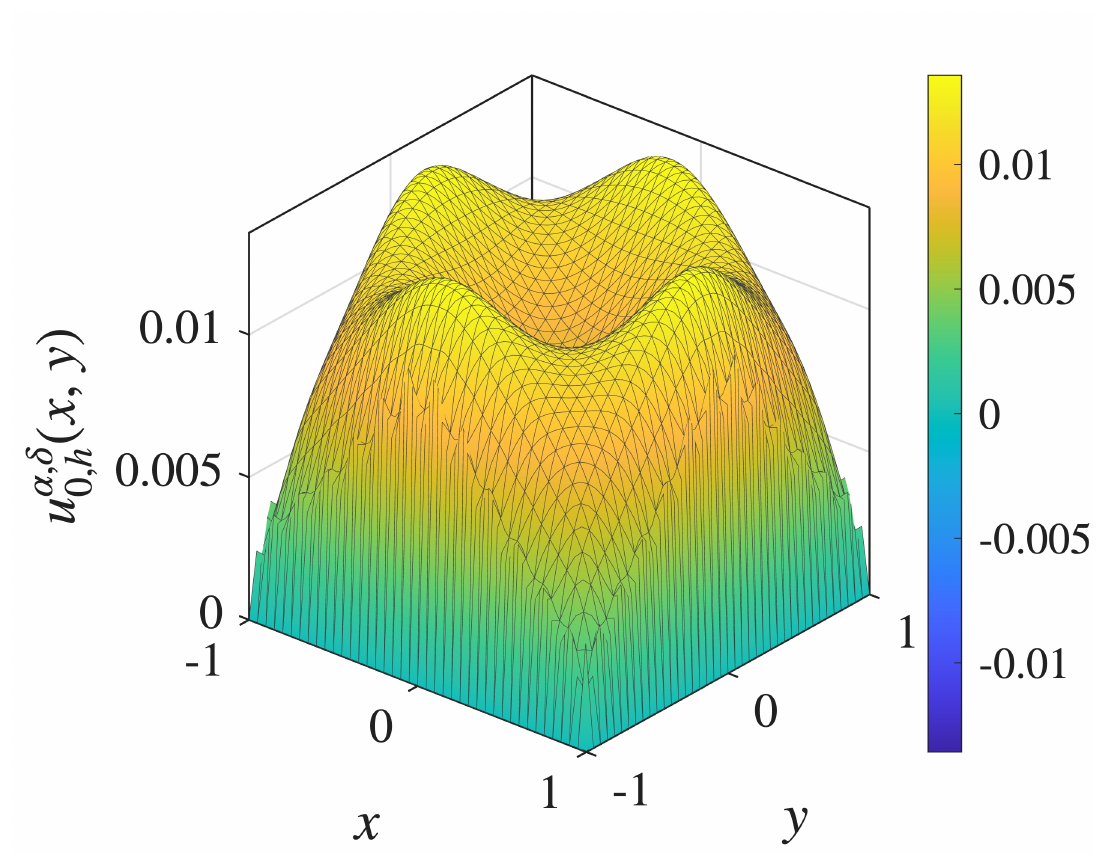}
  \caption{Recovered $u_{0,h}^{\alpha,\delta}$ in $\Omega$ for $\delta=10^{-8}$.}
  \label{fig:ex43-2d-urec}
\end{subfigure}
\caption{Example \ref{sec:ex43-2d} (2D). Recovery of the forward solution for the representative noise level $\delta=10^{-8}$.}
\label{fig:ex43-2d-u}
\end{figure}

\begin{figure}[t]
\centering
\begin{subfigure}[t]{0.48\textwidth}
  \centering
  \includegraphics[width=\linewidth]{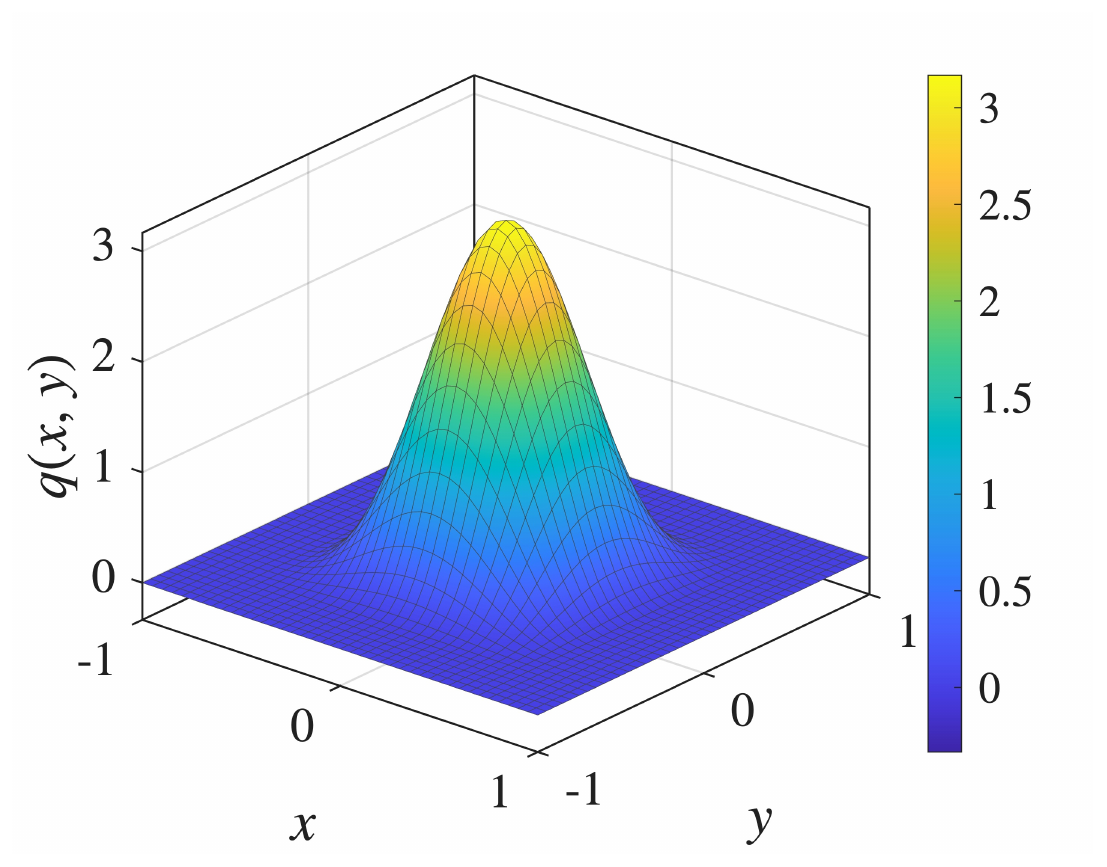}
  \caption{True coefficient $q$ on $\Omega$.}
  \label{fig:ex43-2d-trueq}
\end{subfigure}\hfill
\begin{subfigure}[t]{0.48\textwidth}
  \centering
  \includegraphics[width=\linewidth]{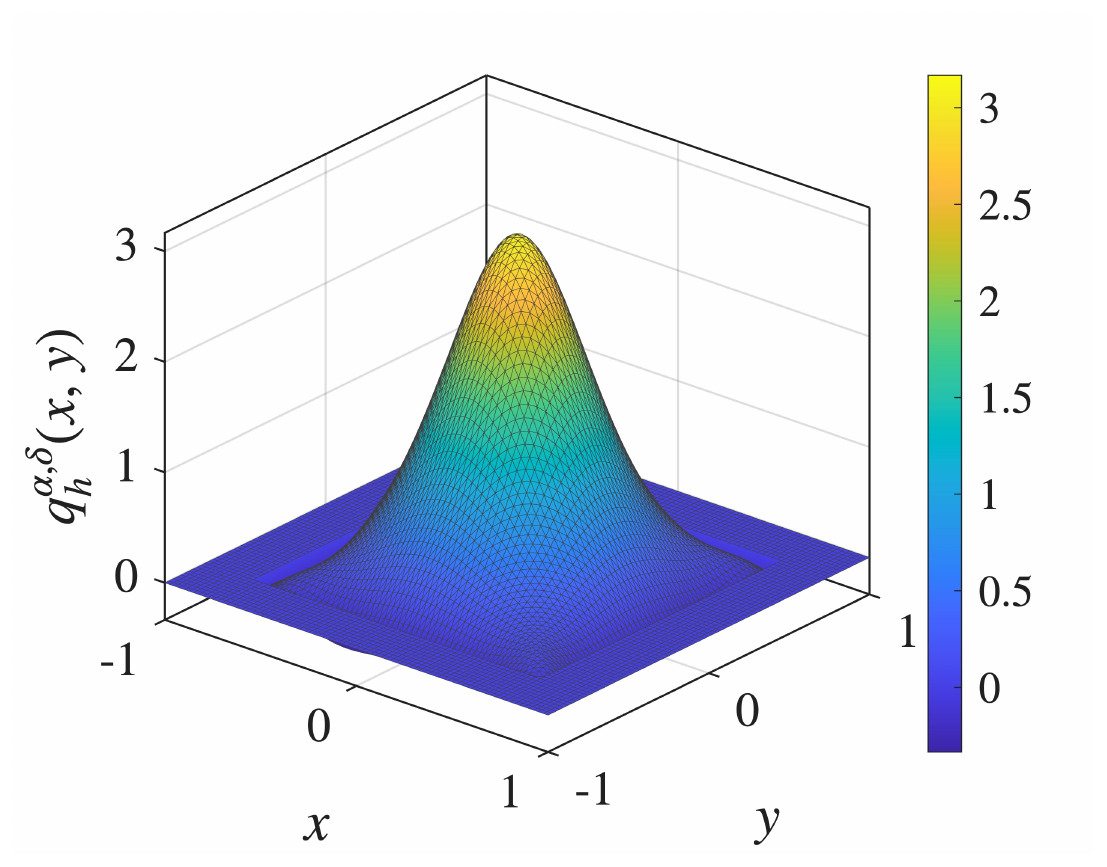}
    \caption{Reconstruction $q_h^{\alpha,\delta}$ for $\delta=10^{-8}$.}
  \label{fig:ex43-2d-qrec}
\end{subfigure}

\vspace{0.4em}

\begin{subfigure}[t]{0.48\textwidth}
  \centering
  \includegraphics[width=\linewidth]{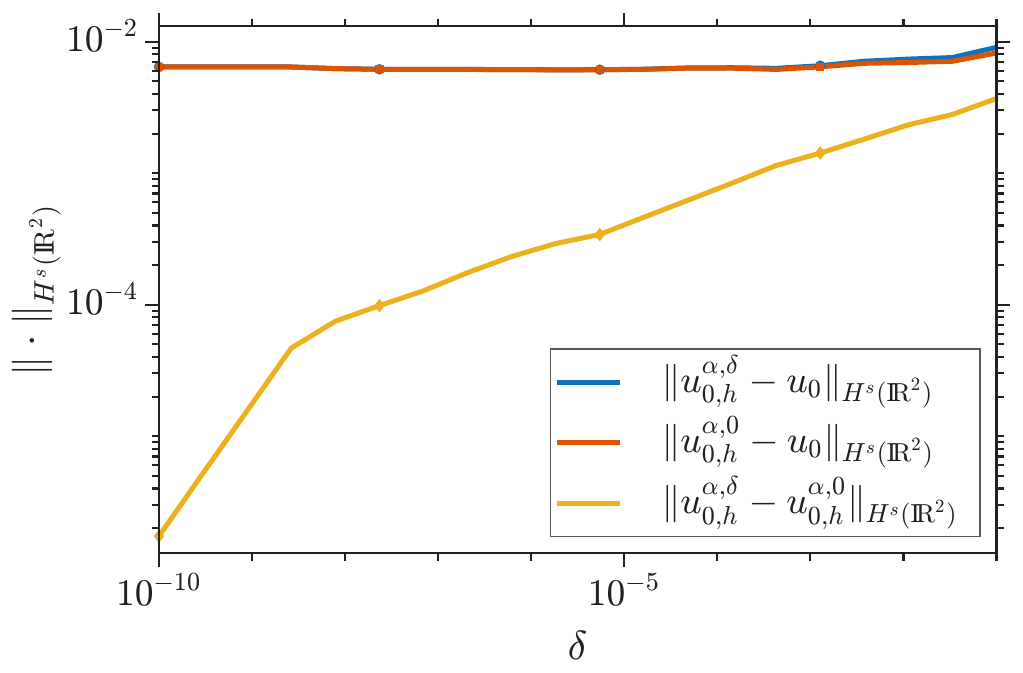}
   \caption{Error decomposition of \(u_0\) in $H^s(\mathbb R^2)$.}
    \label{fig:ex43-2d-state-decomp}
\end{subfigure}\hfill
\begin{subfigure}[t]{0.48\textwidth}
  \centering
  \includegraphics[width=\linewidth]{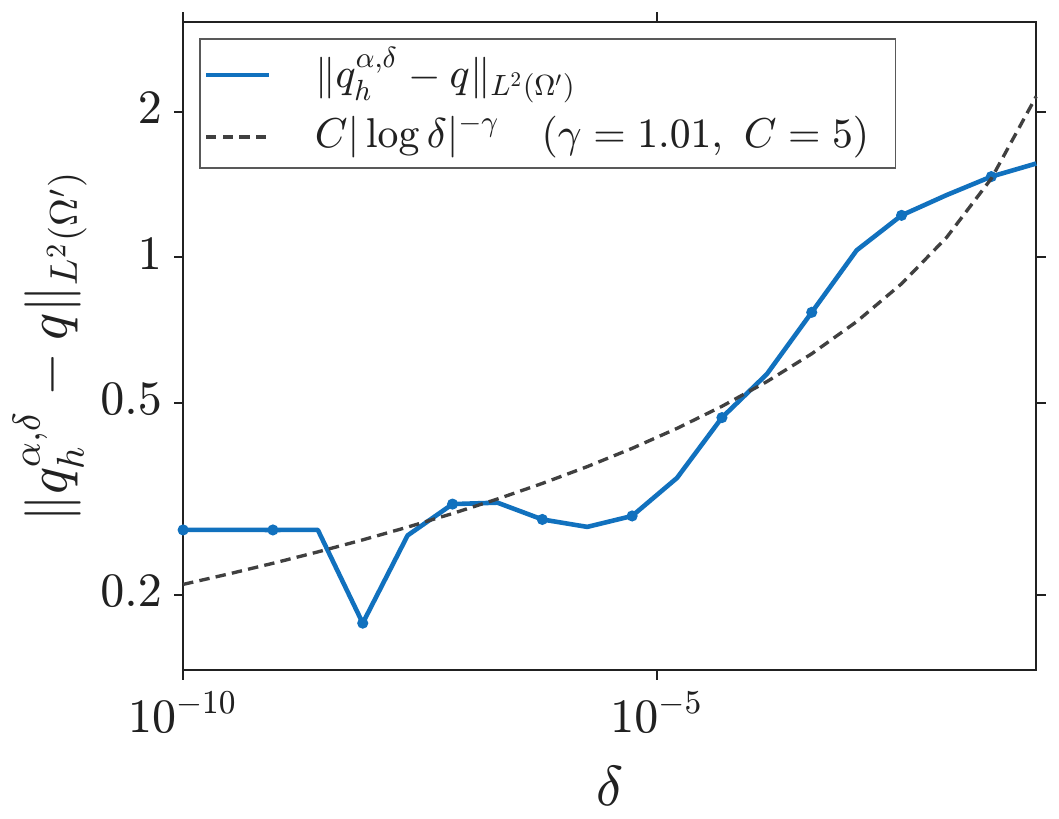}
    \caption{Mean potential error in $L^2(\Omega')$.}
  \label{fig:ex43-2d-stab}
\end{subfigure}
\caption{Example \ref{sec:ex43-2d} (2D bump potential). Setup:
\(\Omega = (-1,1)^2\), \(\Omega_R = (-3,3)^2\),
\(\Omega' = (-0.75,0.75)^2\), \(W=\Omega_R\setminus[-1-d_0,1+d_0]^2\)
with \(d_0=0.05\) and \(s=0.5\). Panel (A) shows the true potential $q$.
Panel (B) shows the recovered potential \(q_h^{\alpha,\delta}\) for \(\delta=10^{-8}\), \(\alpha=0.1\,\delta^{3/2}\), \(\alpha_w=\delta^{2/3}\), and \(\alpha_q=0.01\,\delta\). In (C) and (D), errors for noisy data are averaged over $15$ realizations per noise level.}
\label{fig:ex43-2d-q}
\end{figure}

\section{Conclusion}\label{sec:conclusion}
We presented a two-stage Galerkin method for recovering the forward solution and potential from a single partial exterior measurement. We established convergence and error estimates under the stated assumptions, including a logarithmic coefficient bound without a source condition. Numerical experiments in one and two dimensions illustrate reconstructions of continuous and discontinuous potentials. Future work will study convergence, error estimates, and parameter choice for TV-regularized coefficient recovery using compactness arguments for functions of bounded variation and Bregman distance estimates \cite{acar1994bv,burger2004convergence}.

\subsection*{Acknowledgments}

The authors thank Pu-Zhao Kow and Janne Nurminen for many helpful discussions related to this work. Dwivedi and Rupp have been supported by the Deutsche Forschungsgemeinschaft (DFG, German Research Foundation) -- 577175348. Railo was supported by the Research Council of Finland through the Flagship of Advanced Mathematics for Sensing, Imaging and Modelling (grant number 359183), the Emil Aaltonen Foundation, and the Jenny and Antti Wihuri Foundation.

The authors used generative AI to improve the grammar and clarity of the English text and take full responsibility for the content of the manuscript.

\bibliographystyle{siam}
\bibliography{main}

\end{document}